\documentclass[11pt,reqno]{amsart}
\usepackage{fullpage}
\usepackage[mathscr]{eucal}
\usepackage{mathrsfs}
\usepackage{amsfonts}
\usepackage{amsmath}
\usepackage{color}
\usepackage{amsthm}
\usepackage{amssymb}
\usepackage{latexsym}
\usepackage[all]{xy}
\usepackage[mathscr]{eucal}
\usepackage{palatino}
\usepackage{fullpage}
\usepackage{verbatim}

\theoremstyle{plain}
\newtheorem{thm}[equation]{Theorem}
\newtheorem{cor}[equation]{Corollary}
\newtheorem{prop}[equation]{Proposition}
\newtheorem{lem}[equation]{Lemma}

\theoremstyle{definition}
\newtheorem{defn}[equation]{Definition}

\theoremstyle{remark}

\newtheorem{rem}[equation]{Remark}

\makeatletter
\renewcommand{\subsection}{\@startsection{subsection}{2}{0pt}{-3ex
plus -1ex minus -0.2ex}{-2mm plus -0pt minus
-2pt}{\normalfont\bfseries}} \makeatother

\numberwithin{equation}{subsection}

\newcommand{\Lmod}[1]{#1\text{-}{\mathsf{mod}}}
\newcommand{\LMod}[1]{#1\text{-}{\mathsf{Mod}}}
\newcommand{\RMod}[1]{{\mathsf{Mod}}\text{-}#1}
\newcommand{\QCoh}[1]{\mathsf{QCoh}(#1)}

\newcommand{\idot}{{\:\raisebox{2pt}{\text{\circle*{1.5}}}}}

\DeclareMathOperator{\sym}{\mathrm{Sym}}

\DeclareMathOperator{\Ker}{\mathrm{Ker}}

\DeclareMathOperator{\gr}{\mathrm{gr}}

\DeclareMathOperator{\codim}{\mathrm{codim}}

\newcommand{\beq}{\begin{equation}\label}
\newcommand{\eeq}{\end{equation}}

\newcommand{\iso}{{\;\stackrel{_\sim}{\to}\;}}

\renewcommand{\o}{\otimes }

\newcommand{\erem}{\hfill$\lozenge$\end{rem}\vskip 3pt }

\newcommand{\HH}{{\mathbb H}}

\newcommand{\KZ}{\mathsf{KZ}}

\newcommand{\dd}{{\mathscr{D}}}

\newcommand{\pa}{\partial }

\newcommand{\mc}{\mathcal}

\newcommand{\C}{\mathbb{C}}
\newcommand{\g}{\mathfrak{g}}

\newcommand{\Z}{{\mathbb Z}}
\DeclareMathOperator{\Irr}{\mathrm{Irr}}

\newcommand{\ds}{{\dots}}

\newcommand{\id}{\mathrm{id}}

\newcommand{\Cs}{\C^\times}

\newcommand{\mm}{\ms{M}}

\newcommand{\Vo}{V^{o}}

\newcommand{\mbf}{\mathbf}

\newcommand{\mf}{\mathfrak}
\newcommand{\ra}{\longrightarrow}

\newcommand{\nc}{\newcommand}

\newcommand{\ms}{\mathscr}

\nc{\FV}{\mc{F}_{\mathrm{V}}}
\nc{\grV}{\gr_{\mathrm{V}}}
\nc{\ddUh}{\widehat{\dd_{U}}}

\nc{\A}{\mathbb{A}}
\nc{\Grel}{\mc{G}^{\mathrm{rel}}}
\nc{\Grat}{\mc{G}^{\mathrm{rat}}}
\nc{\Squo}[1]{\A^{(#1)}}
\nc{\twist}{\mathrm{twist}}
\nc{\Cd}{\mc{C}}
\nc{\Span}{\mathrm{Span}}
\nc{\Grass}{\mathrm{Gr}}
\nc{\V}{\mc{V}}
\nc{\Ps}{\mathbb{P}}

\renewcommand{\H}{\mathsf{H}}

\newcommand{\F}{\mc{F}}

\newcommand{\sH}{\mc{H}}

\nc{\Mod}{\mathrm{Mod} \,}
\nc{\grad}{\mathrm{grad}}

\nc{\tbh}{{\widetilde{\bh}}}
\nc{\ras}{\Lambda}
\nc{\Loc}{\mathsf{Loc}}

\renewcommand{\H}{\mathsf{H}}
\nc{\bY}{\overline{\Y}}
\nc{\Hamp}{\mathbb{H}^{\perp}}
\nc{\Ham}{\mathbb{H}}
\nc{\Der}{\mathrm{Der}}
\nc{\Dun}{\mathrm{Dun}}
\nc{\sfB}{\mathsf{B}}
\nc{\sfN}{\mathsf{N}}

\newcommand{\bc}{\mathbf{c}} 
\nc{\by}{\mbf{y}}
\nc{\map}{\varphi}

\begin{document}

\title{Affinity of Cherednik algebras on projective space}

\author{Gwyn Bellamy}

\address{School of Mathematics and Statistics, University of Glasgow, 15 University Gardens, Glasgow, G12 8QW.}
\email{gwyn.bellamy@glasgow.ac.uk}

\author{Maurizio Martino}

\email{ohmymo@googlemail.com}

\begin{abstract}
We give sufficient conditions for the affinity of Etingof's sheaves of Cherednik algebras on projective space. To do this we introduce the notion of pull-back of modules under certain flat morphisms.   
\end{abstract}

\maketitle

{\small
\tableofcontents
}

\section{Introduction}

\subsection{} In the seminal paper \cite{EG}, Etingof and Ginzburg introduced the family of rational Cherednik algebras associated to a complex reflection group. Since their introduction, rational Cherednik algebras have been intensively studied, and found to be related to several other areas of mathematics. Their definition was vastly generalized by Etingof in the preprint \cite{ChereSheaf}. Given any smooth variety $X$ and finite group $W$ acting on $X$, Etingof defines a family of sheaves\footnote{Here, one must take the $W$-equivariant Zariski topology on $X$. See section \ref{sec:conv}.} of algebras $\sH_{\omega,\bc}(X,W)$ on $X$ which are flat deformations of the skew group ring $\dd_X \rtimes W$. Being sheaves of algebras, one would like to be able to use standard geometric techniques such as pull-back and push-forward to study their representation theory. This paper is a small first step in developing these techniques. As motivation, we consider the question of affinity for these algebras when $X = \mathbb{P}(V)$. 

\subsection{} If $V$ is a finite-dimensional vector space and $W$ acts linearly on $V$, then there is an induced action of $W$ on $\mathbb{P}(V)$. Thus, Etingof's construction gives us a sheaf of algebras $\sH_{\omega,\bc}(\mathbb{P}(V),W)$ on $\mathbb{P}(V)$. In trying to understand the representation theory of these algebras, one would like to know when they are affine i.e. for which $\omega$ and $\bc$ does the global sections functor give us an equivalence between the category of modules for $\sH_{\omega,\bc}(\mathbb{P}(V),W)$ and the category of modules for its global sections $\H_{\omega,\bc}(\mathbb{P}(V),W)$. Our main result is an explicit combinatorial criterion on $\omega$ and $\bc$ which guarantees that the corresponding Cherednik algebra is affine. We associate to $\omega, \bc$ and $\lambda \in \Irr W$ a pair of scalars $a_{\lambda},b_{\lambda}$; see section \ref{sec:mainthm}. 

\begin{thm}
The sheaf of algebras $\sH_{\omega,\bc}(\mathbb{P}(V),W)$ is affine provided $a_{\lambda} \notin \Z_{\ge 0}$ and $b_{\lambda} \notin \Z_{>0}$ for all $\lambda \in \Irr W$.
\end{thm}

In order to prove this result, we introduce two key pieces of machinery. The first is the notion of pull-back of $\sH_{\omega,\bc}$-modules under certain well-behaved maps (which we call melys). The second is to establish an equivalence between the category of (twisted) $T$-equivariant $\sH_{\bc}$-modules on a principal $T$-bundle $Y \rightarrow X$ and the category of modules for a Cherednik algebra $\sH_{\omega,\bc}$ on the base $X$ of the bundle. With this machinery in place, the proof of the main result is essentially the same as for sheaves of twisted differential operators on $\mathbb{P}(V)$, see \cite[Theorem 1.6.5]{HTT}. 

\subsection{} Being able to pull-back $D$-modules is an extremely useful tool in studying the representation theory of sheaves of differential operators. Therefore, one would like to be able to do the same for Cherednik algebras. We show that this is possible, at least for some morphisms. A $W$-equivariant map $\map : Y \rightarrow X$ between smooth varieties is said to be \textit{melys} if it is flat and, for all reflections $(w,Z)$ in $X$, $\map^{-1}(Z)$ is contained in the fixed point set $Y^w$ of $w$.

\begin{thm}\label{thm:pullbackintro}
If $\map : Y \rightarrow X$ is melys, then pull-back is an exact functor
$$
\map^* : \LMod{\sH_{\omega,\bc}(X,W)} \ra \LMod{\sH_{\map^* \omega,\map^* \bc}(Y,W)}.
$$
\end{thm}

The pull-back functor is particularly well-behaved when $\map$ is etal\'e. We define the melys site over $X$, a certain modification of the usual etal\'e site over $X$. Using Theorem \ref{thm:pullbackintro}, we show that the Cherednik algebra forms a sheaf on this site. 

One particularly rich source of melys morphisms is when $\pi : Y \rightarrow X$ is a principal $T$-bundle, where $T$ is a torus acting on $Y$ with the action commuting with the action of $W$. In this situation, one can perform quantum Hamiltonian reduction of the Cherednik algebra $\sH_{\bc}(Y,W)$ on $Y$ to get a sheaf $\sH_{\beta(\chi),\bc}(X,W)$ of Cherednik algebras on $X$. As a consequence, one gets an equivalence between the category of ($\chi$-twisted) $T$-equivariant $\sH_{\bc}(Y,W)$-modules on $Y$ and the category of $\sH_{\beta(\chi),\bc}(X,W)$-modules on $X$. 

\begin{thm}\label{thm:monodromicequivintro}
Let $\chi \in \mf{t}^*$. We have an isomorphism of sheaves of algebras on $X$,
$$
\sH_{\beta(\chi),\bc}(X,W) \simeq (\pi_{\idot} \sH_{\bc}(Y,W))^T / \langle \{ t - \chi(t) \ | \ t \in \mf{t} \} \rangle,
$$
and the functor 
$$
\LMod{(\sH_{\bc}(X, W),T,\chi)} \longrightarrow \LMod{\sH_{\beta(\chi),\bc}(Y, W)}
$$
given by $\mm \mapsto (\pi_{\idot} \mm)^T$ is an equivalence of categories, with quasi-inverse $\ms{N} \mapsto \pi^* \ms{N}$. 
\end{thm}

\subsection{} We also study a natural generalization of the Knizhnik-Zamolodchikov connection. The question of whether the Knizhnik-Zamolodchikov connection is flat is closely related to the issue of presenting the Cherednik algebra. In the appendix we summarize, for the reader unfamiliar with sheaves of twisted differential operators, those basic properties of TDOs that we require.

\subsection{Acknowledgments}

The authors would like to thank Collin Ingalls, Anne Shepler and Cedric Bonnaf\'e for fruitful discussions. The first author is supported by the EPSRC grant EP-H028153.

\section{Sheaves of Cherednik algebras}

In this section we introduce sheaves of Cherednik algebras on a smooth variety. 

\subsection{Conventions}\label{sec:conv} Throughout, all our spaces will be equipped with the action of a finite group $W$. The morphisms $\map : Y \rightarrow X$ that we will consider will always be assumed to be $W$-equivariant. Since we wish to deal with objects such as $\mc{O}_X \rtimes W$, we work throughout with the \textit{$W$-equivariant Zariski topology}; a subset $U \subset X$ is an open subset in this topology if and only if it is  open in the Zariski topology and $W$-stable. Then, $\mc{O}_X \rtimes W$ becomes a sheaf on $X$. If $w \in W$, then $X^w$ denotes the set of all points fixed under the automorphism $w$. The sheaf of vector fields, resp. one-forms, on a smooth variety $X$ is denoted $\Theta_X$, resp. $\Omega^1_X$. 

\subsection{} Let $X$ be a smooth, connected, quasi-projective variety over $\C$. Let $Z$ be a smooth subvariety of $X$ of codimension one. Locally, the ideal defining $Z$ is principal, generated by one section, $f_Z$ say. Then, the element  
$$
d \log f_Z := \frac{d f_Z}{f_Z} 
$$
is a section of $\Omega^1_X(Z) = \Omega^1_X \o \mc{O}_X(Z)$. Contraction defines a pairing $\Theta_X \o \Omega^1_X (Z) \rightarrow \mc{O}_X(Z)$, $(\nu,\omega) \mapsto \omega \wedge \nu$. Let $\Omega_X^{1,2}$ be the two term subcomplex $\Omega_X^1 \stackrel{d}{\longrightarrow} (\Omega_X^2)^{\mathrm{cl}}$, concentrated in degrees $1$ and $2$, of the algebraic de-Rham complex of $X$, where $(\Omega_X^2)^{\mathrm{cl}}$ denotes the subsheaf of closed forms in $\Omega_X^2$. As noted in the appendix, sheaves of twisted differential operators on $X$ are parameterized, up to isomorphism, by the second hypercohomology group $\HH^2(X,\Omega_X^{1,2})$. Given $\omega \in \HH^2(X,\Omega_X^{1,2})$, the corresponding sheaf of differential operators is denoted $\dd_{X}^{\omega}$. 

\subsection{Dunkl-Opdam operators} Let $W$ be a finite group acting on $X$. Let $\mc{S}(X)$ be the set of pairs $(w,Z)$ where $w \in W$ and $Z$ is a connected component of $X^w$ of codimension one. Any such $Z$ is smooth. A pair $(w,Z)$ in $\mc{S}(X)$ will be referred to as a \textit{reflection} of $(X,W)$. The group $W$ acts on $\mc{S}(X)$ and we fix $\bc : \mc{S}(X) \rightarrow \C$ to be a $W$-equivariant function. For $\omega \in \HH^2(X,\Omega_X^{1,2})^W$, let $\mc{P}^{\omega}$ be the associated $W$-equivariant Picard algebroid on $X$, with anchor map $\sigma : \mc{P}^{\omega} \rightarrow \Theta_X$. We fix an open affine, $W$-stable covering $\{ U_i \}$ of $X$ such that $\mathrm{Pic} (U_i) = 0$ for all $i$. Then, we can choose functions $f_{Z,i}$ defining $U_i \cap Z$. The union of all $Z$'s is denoted $D$. If $j : X - D \hookrightarrow X$ is the inclusion, then write $\mc{P}^{\omega} (D)$ for the sheaf $j_{\idot} (\mc{P}^{\omega} |_{X - D})$.   

\begin{defn}
To each $\nu \in \Gamma(U_i,\mc{P}^{\omega})$, the associated Dunkl-Opdam operator is 
\beq{eq:DOop}
D_{\nu} = \nu + \sum_{(w,Z) \in \mc{S}(X)} \frac{2 \bc (w,Z)}{1 - \lambda_{w,Z}} (d \log f_{Z,i} \wedge \sigma (\nu)) (w - 1),
\eeq
where $\lambda_{w,Z}$ is the eigenvalue of $w$ on each fiber of the conormal bundle of $Z$ in $X$. 
\end{defn}

The operator $D_v$ is a section of $\mc{P}^{\omega}(D) \rtimes W$ over $U_i$. The $\Gamma(U_i,\mc{O}_X \rtimes W)$-submodule of $\mc{P}^{\omega}(D) \rtimes W$ generated by all Dunkl-Opdam operators $\{ D_v \ | \ v \in \Gamma(U_i,\mc{P}^{\omega}) \}$ and $\Gamma(U_i,\mc{O}_X \rtimes W)$ is denoted $\Gamma(U_i, \mc{F}^1_{\omega,\bc}(X,W))$. Though the definition of the Dunkl-Opdam operator $D_v$ depends on the choice of functions $f_{Z,i}$, it is easy to see that the submodule $\Gamma(U_i, \mc{F}^1_{\omega,\bc}(X,W))$ of $\Gamma(U_i, \mc{P}^{\omega}(D) \rtimes W)$ is independent of any choices. The modules $\Gamma(U_i, \mc{F}^1_{\omega,\bc}(X,W))$ glue to form a sheaf (in the $W$-equivariant Zariski topology!) $\mc{F}^1_{\omega,\bc}(X,W)$ on $X$. As noted in the remark after \cite[Theorem 2.11]{ChereSheaf}, a calculation in each formal neighborhood of $X$ shows that $[D_{\nu_1},D_{\nu_2}] \in \mc{F}^1_{\omega,\bc}(X,W)$ for all $\nu_1, \nu_2 \in \mc{P}^{\omega}$. However, there is \textit{no} natural bracket on $\mc{F}^1_{\omega,\bc}(X,W)$. The anchor map $\sigma : \mc{P}^{\omega}(D) \o W \rightarrow \Theta_X(D) \o W$ restricts to a map $\mc{F}^1_{\omega,\bc}(X,W) \rightarrow \Theta_X \o W$ which fits into a short exact sequence
\beq{eq:ses}
0 \rightarrow \mc{O}_X \rtimes W \rightarrow \mc{F}^1_{\omega,\bc}(X,W) \stackrel{\sigma}{\longrightarrow} \Theta_X \o W \rightarrow 0.
\eeq

\begin{defn}
The subsheaf of algebras of $j_{\idot} (\dd_{X - D}^{\omega} \rtimes W)$ generated by $\mc{F}^1_{\omega,\bc}(X,W)$ is called the \textit{sheaf of Cherednik algebras} associated to $W,\omega$ and $\bc$. It is denoted $\sH_{\omega,\bc}(X,W)$.   
\end{defn}

The global sections of $\sH_{\omega,\bc}(X,W)$ are denoted $\H_{\omega,\bc}(X,W)$.

\subsection{}\label{sec:PBW} There is a natural order filtration $\F^{\idot}_{\omega,\bc}(X,W)$ on $\sH_{\omega,\bc}(X,W)$, defined in one of two ways. Either, one defines $\F^{\idot}_{\omega,\bc}(X,W)$ to be the restriction to $\sH_{\omega,\bc}(X,W)$ of the order filtration on $j_{\idot} (\dd_{X - D}^{\omega} \rtimes W)$. Or, equivalently, by giving elements in $\F^1_{\omega,\bc}(X,W)$ degree at most one, with $D \in \F^1_{\omega,\bc}(X,W)$ having degree one if and only if $\sigma(D) \neq 0$, and then defining the filtration inductively by setting $\F^i_{\omega,\bc}(X,W) = \F^{1}_{\omega,\bc}(X,W)  \F^{i-1}_{\omega,\bc}(X,W)$.  By definition, the filtration is exhaustive. Let $\pi : T^* X \rightarrow X$ be the projection map. Etingof has shown in Theorem 2.11 of \cite{ChereSheaf} that the algebras $ \sH_{\omega,\bc}(X,W)$ are a flat deformation of $\dd_X \rtimes W$. Equivalently, the PBW property holds for Cherednik algebras: 

\begin{thm}\label{thm:PBW}
We have $\gr_{\F} \sH_{\omega,\bc}(X,W) \simeq \pi_{\idot} \mc{O}_{T^* X} \rtimes W$. 
\end{thm}

We note for later use that Theorem \ref{thm:PBW} implies that, for any affine $W$-stable open set $U \subset X$, the algebra $\Gamma(U, \sH_{\omega,\bc}(X,W))$ has finite global dimension; its global dimension is bounded by $2 \dim X$. 

\subsection{} Throughout, a $\sH_{\omega,\bc}(X,W)$-module will always mean a $\sH_{\omega,\bc}(X,W)$-module that is quasi-coherent over $\mc{O}_X$. The category of all $\sH_{\omega,\bc}(X,W)$-modules is denoted $\LMod{\sH_{\omega,\bc}(X,W)}$ and the full subcategory of all modules coherent over $\sH_{\omega,\bc}(X,W)$ is denoted $\Lmod{\sH_{\omega,\bc}(X,W)}$. A module $\ms{M} \in \LMod{\sH_{\omega,\bc}(X,W)}$ is called \textit{liss\'e} if it is coherent over $\mc{O}_X$. 

\section{Pull-back of sheaves}

In this section we show that modules for sheaves of Cherednik algebras can be pulled back under morphisms that are "melys" for the parameter $\bc$. 

\subsection{} Let $\map : Y \rightarrow X$ be a $W$-equivariant morphism between smooth, connected, quasi-projective varieties. As explained in the appendix, given a Picard algebroid $\mc{P}_X^{\omega}$ on $X$, there is a $\map$-morphism $\mc{P}_Y^{\map^* \omega} \rightarrow \map^* \mc{P}_X^{\omega}$. This implies that the sheaf $\map^* \dd_X^{\omega}$ is a left $\dd_Y^{\map^* \omega}$-module. We give conditions on the map $\map$ so that there exists a sheaf of Dunkl operators $\F^1_{\map^* \omega,\map^* \bc}(Y,W)$ on $Y$ and morphism of $\mc{O}_Y \rtimes W$-modules $\F^1_{\map^* \omega,\map^* \bc}(Y,W) \rightarrow \map^* \F^1_{\omega,\bc}(X,W)$. As a consequence $\map^* \sH_{\omega,\bc}(X,W)$ becomes a left $\sH_{\map^* \omega,\map^* \bc}(Y,W)$-module and we can pull-back $\sH_{\omega,\bc}(X,W)$-modules to $\sH_{\map^* \omega,\map^* \bc}(Y,W)$-modules. 

\subsection{} If the morphism $\map$ is flat of relative dimension $r$, then there is a good notion of pull-back of algebraic cycles $\map^* : C_k(X) \rightarrow C_{k+r}(Y)$, where $C_k(X)$ is the abelian group of $k$-dimensional algebraic cycles on $X$. See section 1.7 of \cite{FultonIntersection}. The class  in $C_k(X)$ of a $k$-dimensional subscheme $Z$ of $X$ is denoted $[Z]$. 

\begin{lem}\label{lem:irrcomp}
Let $\map : Y \rightarrow X$ be flat and $(w,Z) \in \mc{S}(X)$. Write $\map^* [Z] = \sum_{i} n_i [Z_i]$, where each $Z_i$ is an irreducible subvariety of $Y$. Then, $w$ permutes the $[Z_i]$. Moreover, if $\map^{-1}(Z)$ is set-theoretically contained in $Y^w$, then each irreducible component of $\map^{-1}(Z)$ is a connected component of $Y^{w}$ of codimension one. 
\end{lem}

\begin{proof}
The first claim follows from the fact that, set-theoretically, $\map^{-1}(Z) = \bigcup_{n_i \neq 0} Z_i$. Since $\map^{-1}(Z)$ is a union of closed subvarieties of $Y$ of codimension one and $Y$ is assumed to be irreducible, it suffices for the second claim to show that $Y^w \neq Y$.  Assume otherwise. Then, since $\map$ is flat, $\map(Y^w) = \map(Y)$ is open in $X$, but also contained in the closed subvariety $X^w$. Hence $X^w = X$. This contradicts the fact that $Z$ is an irreducible component of $X^w$. 
\end{proof}

\subsection{}\label{sec:melysdefn} Let $\mc{S}_{\bc}(X)$ denote the set of all pairs $(w,Z) \in \mc{S}(X)$ such that $\bc (w,Z) \neq 0$. 

\begin{defn}
The morphism $\map$ is \textit{melys} with respect to $\bc$ if:
\begin{enumerate}
\item $\map$ is flat. 
\item For all $(w,Z) \in \mc{S}_{\bc}(X)$, set-theoretically $\map^{-1}(Z) \subset Y^w$. 
\end{enumerate}
\end{defn}

If $\map$ is melys with respect to $\bc$ then we define $\map^*\bc$ on $\mc{S}(Y)$ by 
$$
(\map^* \bc)(w,Z') = \sum_{(w,Z) \in \mc{S}_{\bc}(X)} \langle \map^* [Z], [Z'] \rangle \ \bc(w,Z).
$$
Let $E = \bigcup_{\bc (w,Z) \neq 0} Z$ and $D = \map^{-1}(E)$. Since $\map$ is flat, each irreducible component of $D$ has codimension one in $X$. Let $j : U := X - D \hookrightarrow X$ and $k : V = Y - E \hookrightarrow Y$; these are affine morphisms. For any quasi-coherent sheaf $\mc{F}$ on $X$ (resp. on $Y$), we denote by $\mc{F}(D)$ the sheaf $j_{\idot} (\mc{F} |_U)$ (resp. by $\mc{F}(E)$ the sheaf $k_{\idot} (\mc{F} |_V)$).  

\begin{lem}\label{lem:twistedgammaW}
The sheaf $\map^* \dd^{\omega}_Y(E) \rtimes W$ on $X$ is a $\dd^{\map^* \omega}_X(D) \rtimes W$-module and there exists a morphism 
$$
\gamma : \dd^{\map^* \omega}_X(D) \rtimes W \longrightarrow \map^* \dd^{\omega}_Y(E) \rtimes W
$$
of $\dd^{\map^* \omega}_X(D) \rtimes W$-modules. 
\end{lem}

\begin{proof}
The map $\map$ restricts to a flat morphism $\Phi : U \rightarrow V$. By Lemma \ref{lem:fmorphism}, we have 
$$
\mc{P}_U^{\Phi^* \omega} \stackrel{\sim}{\longrightarrow} \Phi^* \mc{P}^{\omega}_V \times_{\Phi^* \Theta_V } \Theta_U.
$$
This induces a morphism $\gamma : \dd^{\Phi^* \omega}_U \rightarrow \Phi^* \dd^{\omega}_V$ of $\dd^{\Phi^* \omega}_U$-modules. Since $\omega$ was chosen to be $W$-invariant, this extends to a morphism $\gamma : \dd^{\Phi^* \omega}_U \rtimes W \rightarrow \Phi^* \dd^{\omega}_V \rtimes W$ of $\dd^{\Phi^* \omega}_U \rtimes W$-modules. Since $j_{\idot} \mc{P}_U^{\Phi^* \omega} = \mc{P}_X^{\map^* \omega}(D)$, we have $j_{\idot}( \dd^{\Phi^* \omega}_U \rtimes W) =  \dd^{\map^* \omega}_X(D) \rtimes W$. The diagram 
$$
\xymatrix{
U \ar@{^{(}->}[r]^j \ar[d]_{\Phi} & X \ar[d]^\map \\
V \ar@{^{(}->}[r]^k & Y
}
$$
is Cartesian. Therefore, by flat base change, $j_{\idot} \Phi^* \mc{P}^{\omega}_V \rtimes W = \map^* \mc{P}^{\omega}_Y(E) \rtimes W$ and hence $j_{\idot} (\Phi^* \dd^{\omega}_V \rtimes W) = \map^* \dd^{\omega}_Y(E) \rtimes W$. 
\end{proof}

\subsection{} By analogy with $\map$-morphisms, see Lemma \ref{lem:fmorphism}, we have

\begin{prop}\label{lem:chainrule}
There is a morphism 
$$
\gamma : \sH_{\map^* \omega, \map^* \bc}(Y,W) \longrightarrow \map^* \sH_{\omega, \bc}(X,W)
$$
of $\sH_{\map^* \omega, \map^* \bc}(Y,W)$-modules, that induces an isomorphism of $\mc{O}_Y \rtimes W$-modules 
$$
\psi : \F^1_{\map^* \omega, \map^* \bc}(Y,W) \stackrel{\sim}{\longrightarrow} \map^* \F^1_{\omega, \bc}(X,W) \times_{\map^* \Theta_X \o W} \Theta_Y \o W.
$$ 
\end{prop}

\begin{proof}
The algebra $\sH_{\map^* \omega, \map^* \bc}(Y,W)$ is a subalgebra of $\dd^{\map^* \omega}_Y(E) \rtimes W$ and $\map^* \sH_{\omega, \bc}(X,W)$ is a subalgebra of $\map^* \dd^{\omega}_X(D) \rtimes W$. Let $\gamma : \sH_{\map^* \omega, \map^* \bc}(Y,W) \rightarrow \map^* \dd^{\omega}_X(D) \rtimes W$ be the restriction of the morphism $\gamma : \dd^{\map^* \omega}_Y(E) \rtimes W \rightarrow \map^* \dd^{\omega}_X(D) \rtimes W$ of Lemma \ref{lem:twistedgammaW}. We claim that it suffices to show that the image of $\gamma$ is contained in $\map^* \sH_{\omega, \bc}(X,W)$. Assuming this, the action of $\sH_{\map^* \omega, \map^* \bc}(Y,W)$ on $\map^* \sH_{\omega, \bc}(X,W)$ will just be the restriction of the action of $\dd^{\map^* \omega}_Y(E) \rtimes W$ on $\map^* \dd^{\omega}_X(D) \rtimes W$. Therefore, it is given by 
$$
a \cdot (g \o p) = \gamma([a,g]) \cdot (1 \o p) + g (\gamma(a) \cdot (1 \o p)), 
$$
where $a \in \sH_{\map^* \omega, \map^* \bc}(Y,W)$, $g \in \mc{O}_Y$ and $p \in \map^{-1} \sH_{\omega, \bc}(X,W)$. Here $[a,g]$ is thought of as an element of $\sH_{\map^* \omega, \map^* \bc}(Y,W)$. If $\gamma(a)$ is contained in $\map^* \sH_{\omega, \bc}(X,W)$ and $p \in \map^{-1} \sH_{\omega, \bc}(X,W)$, then $\gamma(a) \cdot (1 \o p)$ belongs to $\map^* \sH_{\omega, \bc}(X,W)$. Thus, it suffices to show that the image of $\gamma$ is contained in $\map^* \sH_{\omega, \bc}(X,W)$ as claimed. 

Since $\sH_{\map^* \omega, \map^* \bc}(Y,W)$ is generated as an algebra by $\F^1_{\map^* \omega, \map^* \bc}(Y,W)$, it will suffice to show that the image of $\F^1_{\map^* \omega, \map^* \bc}(Y,W)$ is contained in $\map^* \F^1_{ \omega, \bc}(X,W)$. This is a local calculation. Therefore, we may assume that both $X$ and $Y$ are affine and that the subvarieties $Z$ of $X$ with $(w,Z) \in \mc{S}_{\bc}(X)$ are defined by the vanishing of functions $f_Z$. Let $p \in \mc{P}^{\map^* \omega}_Y$ and $D_p$ the associated Dunkl-Opdam operator as given by (\ref{eq:DOop}). Let $\gamma(p) = \sum_i g^i \o q^i$ in $\map^* \mc{P}^{\omega}_X$. Then, 
$$
\gamma(D_p) = \sum_i g^i \o q^i + \sum_{(w,Z')}  \frac{2 (\map^* \bc) (w,Z')}{1 - \lambda_{w,Z'}} (d \log f_{Z'} \wedge \sigma_Y(p)) \o (w - 1).
$$
If $\map^{-1}(Z) = Z_1' \cup \cdots \cup Z_l'$ set-theoretically and $n_i = \langle [Z_i'],\map^*[Z] \rangle$, then $\map^* f_Z = u \prod_i f_{Z_i'}^{n_i}$, for some unit $u$, and scheme-theoretically $\map^{-1}(Z)$ is defined by the vanishing of the function $\prod_i f_{Z_i'}^{n_i}$. Therefore, by definition of the parameter $\map^* \bc$, 
\beq{eq:pullbackform}
\frac{2 \bc (w,Z)}{1 - \lambda_{w,Z}} \map^* d \log f_{Z} = \sum_{Z' \subset \map^{-1}(Z)} \frac{2 (\map^* \bc) (w,Z')}{1 - \lambda_{w,Z'}} d \log f_{Z'} + h,
\eeq
where $h \in \mc{O}_Y \rtimes W$. Hence, up to a term in $\map^* \mc{O}_X \rtimes W$, 
\begin{align*}
& \sum_{(w,Z')}  \frac{2 (\map^* \bc) (w,Z')}{1 - \lambda_{w,Z'}} (d \log f_{Z'} \wedge \sigma_Y(p)) \o (w - 1) = & \\
& \sum_{(w,Z)}  \frac{2 \bc (w,Z)}{1 - \lambda_{w,Z}} (d \log \map^* f_{Z} \wedge \sigma_Y(p)) \o (w - 1) = & \\
& \sum_{(w,Z)}  \frac{2 \bc (w,Z)}{1 - \lambda_{w,Z}} \frac{\sigma_Y(p)(\map^* f_{Z})}{\map^* f_{Z}} \o (w - 1) = & \\ 
& \sum_{(w,Z)}  \frac{2 \bc (w,Z)}{1 - \lambda_{w,Z}} \frac{1}{\map^* f_{Z}} \left( \sum_i g^i \map^* ( \sigma_X(q^i)(f_{Z})) \right) \o (w - 1) = & \\ 
& \sum_i g^i \o \left( \sum_{(w,Z)}  \frac{2 \bc (w,Z)}{1 - \lambda_{w,Z}} \frac{\sigma_X(q^i)(f_{Z})}{f_{Z}} (w - 1) \right)  = & \\ 
& \sum_i g^i \o \left( \sum_{(w,Z)}  \frac{2 \bc (w,Z)}{1 - \lambda_{w,Z}} (d \log f_Z \wedge \sigma_X(q^i) ) (w - 1)\right). & 
\end{align*}
Thus, $\gamma(D_p) = \sum_i g^i \o D_{q^i}$, which lies in $\map^* \F^1_{ \omega, \bc}(X,W)$. 

Finally, we show that the morphism $\gamma$ induces the isomorphism $\psi$, as stated. Since $\map$ is flat, pulling back the sequence (\ref{eq:ses}) gives a short exact sequence 
$$
0 \rightarrow \mc{O}_Y \rtimes W \rightarrow \map^* \F^1_{ \omega, \bc}(X,W) \rightarrow \map^* \Theta_X \o W \rightarrow 0.
$$
Using the fact that $\mc{O}_Y \rtimes W \times_{\map^* \Theta_X \o W} \Theta_Y \o W = \mc{O}_Y \rtimes W$, where $\mc{O}_Y \rtimes W \rightarrow \map^* \Theta_X \o W$ is the zero map, and the fact that $\map^* \Theta_X \o W \times_{\map^* \Theta_X \o W} \Theta_Y \o W =\Theta_Y \o W$, we have a commutative diagram 
$$
\xymatrix{
0 \ar[r] & \mc{O}_Y \rtimes W \ar[r] \ar@{=}[d] & \F^1_{\map^* \omega, \map^* \bc}(Y,W) \ar[r]^{\sigma} \ar[d]_{\psi} & \Theta_Y \o W \ar[r] \ar@{=}[d] & 0 \\
0 \ar[r] & \mc{O}_Y \rtimes W \ar[r] & \map^* \F^1_{\omega, \bc}(X,W) \times_{\map^* \Theta_X \o W} \Theta_Y \o W \ar[r] & \Theta_Y \o W \ar[r] & 0. 
}
$$ 
By the five lemma, $\psi$ is an isomorphism. 
\end{proof}

\subsection{} The morphism $\gamma$ allows us to define an action of $\sH_{\map^* \omega,\map^* \bc}(Y,W)$ on $\map^* \mm$, for any $\sH_{\omega,\bc}(X,W)$-module $\mm$. 

\begin{cor}\label{cor:melyspullback}
Assume that $\map$ is melys with respect to $\bc$. Then pull-back is an exact functor
$$
\map^* : \LMod{\sH_{\omega,\bc}(X,W)} \ra \LMod{\sH_{\map^* \omega,\map^* \bc}(Y,W)},
$$
extending the usual pullback $\map^* : \QCoh{X} \ra \QCoh{Y}$.  
\end{cor}

\begin{proof}
Proposition \ref{lem:chainrule} implies that  
$$
\map^* \mm = \map^* \sH_{\omega, \bc}(X,W) \o_{\map^{-1} \sH_{\omega, \bc}(X,W)} \map^{-1} \mm
$$
is naturally a $\sH_{\map^* \omega, \map^* \bc}(Y,W)$-module. Since $\map$ is flat, pull-back of quasi-coherent $\mc{O}_X$-modules is an exact functor. 
\end{proof}

It is clear from the definition that $\map^*$ maps $\Lmod{\sH_{\omega,\bc}(X,W)}$ to $\Lmod{\sH_{\map^* \omega,\map^* \bc}(Y,W)}$ and liss\'e $\sH_{\omega,\bc}(X,W)$-modules to liss\'e $\sH_{\map^* \omega,\map^* \bc}(Y,W)$-modules. 

\begin{rem}
It seems plausible that the notion of pull-back along melys morphisms could be useful in studying the representation theory of rational Cherednik algebras. We hope to return to this in future work. Also, it is possible to give a more general definition of melys morphisms, which takes into account the fact that $W$ may not be acting faithfully on $X$. However, this is a bit more involved. Details will appear elsewhere. 
\end{rem} 


\subsection{Etal\'e morphisms}

In this section we consider etal\'e morphisms. Fix $X$, $W$ $\omega$ and $\bc$ as above. Let $(X,\bc)_{\mathrm{mel}}$ be the full subcategory of $\mathrm{Sch} / X$, schemes over $X$, consisting of all morphisms $Y \rightarrow X$ that are etal\'e and melys with respect to $\bc$. Then, one can easily check that $(X,\bc)_{\mathrm{mel}}$ is a site over $X$; see e.g. \cite[Section II.1]{Milne} for details on sites. We call $(X,\bc)_{\mathrm{mel}}$ the melys site over $X$. The following result is closely related to \cite[Proposition 2.3]{Wilcox}. . 

\begin{prop}\label{prop:etalesheaf}
The sheaf $\sH_{\omega, \bc}(X,W)$ is a sheaf of algebras on the melys site $(X,\bc)_{\mathrm{mel}}$. 
\end{prop}

\begin{proof}
Let $\map : Y \rightarrow X$ be an etal\'e map, melys with respect to $\bc$. We begin by showing that $\map^* \sH_{\omega,\bc}(X,W)$ is a sheaf of algebras and the morphism $\gamma$ of Proposition \ref{lem:chainrule} is an isomorphism of algebras. 

As in section \ref{sec:melysdefn}, let $D = \bigcup Z$, $E = \map^{-1}(D)$, $U = X - D$ and $V = Y - E$. Since $\Phi : V \rightarrow U$ is etal\'e, it is flat and hence $\Phi^{-1} \dd_U^{\omega} \rtimes W$ is a subsheaf of $\Phi^* \dd_U^{\omega} \rtimes W$. As noted in remark \ref{rem:etalepullback}, the natural map $\gamma : \dd^{\Phi^* \omega}_V \rtimes W \rightarrow \Phi^* \dd_U^{\omega} \rtimes W$ is an algebra isomorphism such that the restriction of $\gamma^{-1}$ to $\Phi^{-1} \dd_U^{\omega} \rtimes W$ is an algebra morphism $\Phi^{-1} \dd_U^{\omega} \rtimes W \rightarrow \dd^{\Phi^* \omega}_V \rtimes W$. Therefore, using flat base change as in the proof of Lemma \ref{lem:twistedgammaW}, we get an algebra morphism $\gamma^{-1} : \map^{-1} \dd^{\omega}_X(D) \rtimes W \rightarrow \dd^{\map^* \omega}_Y (E) \rtimes W$. This morphism induces an algebra isomorphism 
$$
\gamma^{-1} : \map^{*} \dd^{\omega}_X(D) \rtimes W \stackrel{\sim}{\longrightarrow} \dd^{\map^* \omega}_Y (E) \rtimes W,
$$
where the multiplication in $\map^{*} \dd^{\omega}_X(D) \rtimes W$ is given by 
$$
(g_1 \o q_1)  \cdot (g_2 \o q_2) = (g_1 \o 1)  u(q_1,g_2) (1 \o q_2)
$$
with $u(q,g) := \gamma([\gamma^{-1}(q),g]) \in \Phi^* \dd^{\omega}_X(D) \rtimes W$, for all $g, g_1,g_2 \in \mc{O}_Y$ and $q, q_1, q_2 \in \map^{-1} \dd^{\omega}_X(D) \rtimes W$. By Proposition \ref{lem:chainrule}, $\gamma^{-1}$ restricts to an algebra morphism $\map^{-1} \sH_{\omega,\bc}(X,W) \rightarrow \sH_{\map^* \omega, \map^* \bc}(Y,W)$, inducing an isomorphism $\map^*  \sH_{\omega,\bc}(X,W) \iso \sH_{\map^* \omega, \map^* \bc}(Y,W)$. Let 
$$
\xymatrix{
Y_1 \ar[rr]^{\vartheta} \ar[rd]_{\map_1} & & Y_2 \ar[dl]^{\map_2} \\
 & X & 
}
$$
be a morphism in $(X,\bc)_{\mathrm{mel}}$. Then, $Y_1$ and $Y_2$ are smooth varieties and, by \cite[I, Corollary 3.6]{Milne}, $\vartheta$ is an etal\'e morphism. Lemma \ref{lem:irrcomp} implies that it is also melys. Thus, the above computations shows that $\sH_{\omega,\bc}(X,W)$ forms a presheaf on $(X,\bc)_{\mathrm{mel}}$. 

To check that it is in fact a sheaf, it suffices to do so locally; see the proof of \cite[Proposition 0]{BorhoBryII}. Therefore, we assume that $X$ is affine and that we are given a smooth etal\'e, $W$-equivariant, affine covering $(i_{\alpha} : Y_{\alpha} \ra X)$ of $X$ i.e. each $Y_{\alpha}$ is affine and the union of the images of the maps $i_\alpha$ cover $X$. Then we must prove that the sequence
$$
0 \ra \H_{\omega,\bc}(X,W) \ra \bigoplus_{\alpha} \H_{i_{\alpha}^* \omega, i^*_{\alpha} \bc}(Y_{\alpha},W) \ra \bigoplus_{\alpha,\beta} \H_{i^*_{\alpha,\beta} \omega, i^*_{\alpha,\beta} \bc}(Y_{\alpha} \times_X Y_{\beta},W)
$$
is exact. Let $U, V_{\alpha}, \ds$ be the usual open subsets of $X, Y_{\alpha}, \ds$.  Then, we have a commutative diagram
$$
\xymatrix{
0 \ar[r] & \H_{\omega,\bc}(X,W) \ar[r]^{j} \ar@{^{(}->}[d] & \bigoplus_{\alpha} \H_{i_{\alpha}^* \omega, i^*_{\alpha} \bc}(Y_{\alpha},W) \ar[r]^k \ar@{^{(}->}[d] & \bigoplus_{\alpha,\beta} \H_{i^*_{\alpha,\beta} \omega, i^*_{\alpha,\beta} \bc}(Y_{\alpha} \times_X Y_{\beta},W) \ar@{^{(}->}[d] \\
0 \ar[r] & \Gamma(U,\dd_U^{\omega} \rtimes W) \ar[r] & \bigoplus_{\alpha} \Gamma(V_{\alpha},\dd^{i_{\alpha}^* \omega}_{V_{\alpha}} \rtimes W) \ar[r] & \bigoplus_{\alpha,\beta} \Gamma(V_{\alpha} \times_U V_{\beta}, \dd^{i^*_{\alpha,\beta} \omega} \rtimes W). 
}
$$
The bottom row is exact because $\dd_U^{\omega} \rtimes W$ is a sheaf on the melys site. Since the diagram commutes, $j$ is injective and the image of $j$ is contained in the kernel of $k$. Therefore, we just need to show that the image of $j$ is exactly the kernel of $k$. The sequence on the bottom row is strictly filtered with respect to the order filtration and, as noted in section \ref{sec:PBW}, the Cherednik algebra inherits its natural filtration by restriction of the order filtration on $\dd_U^{\omega} \rtimes W$. Therefore, the top row will be exact if and only if the corresponding sequence of associated graded objects is exact. But this sequence is also the associated graded of the analogous sequence for $\dd_X \rtimes W$, which we know is exact. 
\end{proof}

\subsection{The $\KZ$-functor} Assume that $W$ acts freely on $V$ and $W$, and let $\omega = 0$. The proof of Proposition \ref{lem:chainrule} makes it clear that pull-back of $\sH_{\bc}(X,W)$-modules is compatible with the $\KZ$-functor. Denote by $\sH_{\bc}(X,W)\text{-}{\mathsf{Reg}}$ the full subcategory of $\Lmod{\sH_{\bc}(X,W)}$ consisting of all liss\'e $\sH_{\bc}(X,W)$-modules whose restriction to $U$ is an integrable connection with regular singularities. Let $\mathsf{DR}$ be the de-Rham functor that maps integrable connections with regular singularities on $U / W$ to representations of the fundamental group $\pi_1 (U / W)$. The $\KZ$-functor is defined by 
$$
\KZ_X (\mm) = \mathsf{DR} \left( [ \rho_{\idot} (\mm |_{U}) ]^W \right).
$$
Then $\map^*$ maps $\sH_{\bc}(X,W)\text{-}{\mathsf{Reg}}$ into $\sH_{\map^* \bc}(Y,W)\text{-}{\mathsf{Reg}}$. Therefore, since the de-Rham functor behaves well with respect to pull-back \cite[Theorem 7.1.1]{HTT}, the following diagram commutes
$$
\xymatrix{
\sH_{\bc}(X,W)\text{-}{\mathsf{Reg}} \ar[r]^{\map^*} \ar[d]_{\KZ_X} & \sH_{\map^* \bc}(Y,W)\text{-}{\mathsf{Reg}} \ar[d]^{\KZ_Y} \\
\Lmod{\pi_1(U/W)} \ar[r]^{\Phi^*} & \Lmod{\pi_1(V/W)}.
}
$$
The image of the $\KZ$-functor is contained in the full subcategory of $\Lmod{\pi_1(U/W)}$ consisting of all modules for a certain "Hecke" quotient of $\C \pi_1(U/W)$; see \cite[Proposition 3.4]{ChereSheaf}. 

\subsection{Push-forward}\label{sec:pushforward} It is also possible to define (derived) push-forward of modules under melys maps. Let $\map : Y \rightarrow X$ be melys with respect to $\bc$ and denote by $\RMod{\sH_{\omega,\bc}(Y,W)}$ the category of right $\sH_{\omega,\bc}(Y,W)$-modules. Then, the derived push-forward functor 
$$
\mathbb{R} \map_* :  D^b(\RMod{\sH_{\omega,\bc}(Y,W)}) \longrightarrow D^b(\RMod{\sH_{\omega,\bc}(X,W)})
$$
is given by 
$$
\mathbb{R} \map_*(\mm) = \mathbb{R} \map_{\idot} \left( \mm \o^{\mathbb{L}}_{\sH_{\omega,\bc}(Y,W)} \map^* \sH_{\omega,\bc}(X,W) \right). 
$$
Let us justify the fact that the image of $\mathbb{R} \map_*$ is contained in $D^b(\RMod{\sH_{\omega,\bc}(X,W)})$. Firstly, as noted in section \ref{sec:PBW}, the PBW Theorem implies that the sheaf $\sH_{\omega,\bc}(Y,W)$ has good homological properties. Since we have assumed that $Y$ is quasi-projective, this implies that each $\mm \in \RMod{\sH_{\omega,\bc}(Y,W)}$ has a finite resolution by locally projective $\sH_{\omega,\bc}(Y,W)$-modules, see \cite[section 1.4]{HTT}. Hence, for $\mm \in  D^b(\RMod{\sH_{\omega,\bc}(Y,W)})$, the complex $\mm \o^{\mathbb{L}}_{\sH_{\omega,\bc}(Y,W)} \map^* \sH_{\omega,\bc}(X,W)$ belongs to $D^b(\RMod{\map^{-1} \sH_{\omega,\bc}(Y,W)})$. Then the fact that $\mathbb{R} \map_*(\mm)$ belongs to $D^b(\RMod{\sH_{\omega,\bc}(X,W)})$ follows, for instance, from \cite[Proposition 1.5.4]{HTT}. 

We will also require push-forward of left $\sH_{\omega,\bc}(Y,W)$-modules under open embeddings $j : Y \hookrightarrow X$. The following is standard, see e.g. \cite[Proposition 1.5.4]{HTT}. 

\begin{lem}\label{lem:opendirect}
For $\mm \in \LMod{\sH_{\omega,\bc}(Y,W)}$, the sheaves $\mathbb{R}^i j_{\idot} (\mm)$ belong to $\LMod{\sH_{\omega,\bc}(X,W)}$ for all $i \ge 0$. 
\end{lem}

It would be interesting to develop a notion of duality for Cherednik algebras, which would allow one to define push-forward of left $\sH_{\omega,\bc}(Y,W)$-modules along arbitrary melys morphisms. 

\section{Twisted equivariant modules}

In this section we define (twisted) $G$-equivariant $\sH_{\omega,\bc}(X,W)$-modules. 

\subsection{} Let $X$ be a smooth $W$-variety and $\sH_{\omega,\bc}(X,W)$ a sheaf of Cherednik algebras on $X$. Assume that a \textit{connected} algebraic group $G$ also acts on $X$ such that this action commutes with the action of $W$. Write $p,a : G \times X \ra X$ for the projection and action maps. Let $\ms{M}$ be a $\sH_{\omega,\bc}(X,W)$-modue. Clearly, $p^* \ms{M}$ is a $\sH_{\omega,\bc}(G \times X, W) = \dd_G \boxtimes \sH_{\omega,\bc}(X,W)$-module. 

\begin{lem}
The action map $a$ is melys for any $\bc$ and hence $a^* \ms{M}$ is a $\sH_{\omega,\bc}(G \times X, W)$-module. 
\end{lem}

\begin{proof}
The action map $a$ is smooth and hence flat. Let $(w,Z) \in \mc{S}(X)$. Since the action of $G$ commutes with the action of $W$, $X^w$ is $G$-stable. Moreover, the fact that $G$ and $Z$ are connected implies that $Z$ itself is $G$-stable. Thus, $a^{-1}(Z) = G \times Z$ is contained in $(G \times X)^w = G \times X^w$. 
\end{proof}

\subsection{} The Lie algebra of $G$ is denoted $\g$. Let $m : G \times G \to G$ the multiplication map and $s : X \to G \times X$, $s(x) = (e,x)$. Choose $\chi \in (\g / [\g,\g])^*$ and let $\mc{O}_G^{\chi}$ be the $\dd_G$-module $\dd_G / \dd_G \{ v - \chi(v) \ | \ v \in \g \}$, where we have identified $\g$ with right invariant vector fields on $G$. It is a simple integrable connection on $G$. 

\begin{defn}
The module $\mm \in \LMod{\sH_{\omega,\bc}(X, W)}$ is said to be $(G,\chi)$-equivariant if there exists an isomorphism $\theta : \mc{O}^{\chi}_G \boxtimes \mm \iso a^* \mm$ of $\sH_{\omega,\bc}(G \times X, W)$-modules such that $s^* \theta = \id_{\mm}$ and the diagram 
\beq{eq:cocycle}
\xymatrix{
\mc{O}^{\chi}_G \boxtimes \mc{O}^{\chi}_G \boxtimes \mm \ar[rr]^{\id_{G} \times \theta} \ar[d]_{=} & & \mc{O}^{\chi}_G \boxtimes a^* \mm \ar[d]^{=} \\
(m \times \id)^* (\mc{O}_G^{\chi} \boxtimes \mm) \ar[dd]_{(m \times \id_X)^* \theta} & & (\id_G \times a)^*(\mc{O}_G^{\chi} \boxtimes \mm) \ar[ddl]^{(\id_G \times a)^* \theta} \\
 & & \\
(m \times \id_X)^* a^* \mm \ar[r]^= & (\id_G \times a)^*a^* \mm  & 
}
\eeq
is commutative ("$\mm$ satisfies the \textit{cocycle} condition"). 
\end{defn}

The category of $(G,\chi)$-equivariant $\sH_{\omega,\bc}(X, W)$-modules is denoted $\LMod{(\sH_{\omega,\bc}(X, W),G,\chi)}$. 

\subsection{$T$-monodromic modules}

Let $T$ be a torus i.e. a product of copies of the multiplicative group $\Cs$. The Lie algebra of $T$ is denoted $\mf{t}$. Let $\pi : Y \rightarrow X$ be a principal $T$-bundle, with $X$ smooth. We assume that the finite group $W$ acts on $Y$, the action commuting with the action of $T$. This implies that $W$ also acts on $X$ and the map $\pi$ is $T$-equivariant. Let $\sH_{\bc}(Y,W)$ be a sheaf of Cherednik\footnote{We assume, for simplicity, that the twist $\omega$ is zero. Presumably one can also deal with non-trivial twists.} algebras on $Y$. 

\begin{lem}
There is a morphism of Lie algebras $\mu_{\bc} : \mf{t} \rightarrow \F^1_{\bc}(Y,W)$ such that the composite $\sigma \circ \mu_{\bc}$ equals the usual moment map $\mu : \mf{t} \rightarrow \Theta_Y \o W$. 
\end{lem}

\begin{proof}
Since the action of $T$ commutes with the action of $W$, the open set $V = Y - E$ is $T$-stable. Differentiating the action of $T$ on $U$, there is a map $\mu' : \mf{t} \rightarrow \dd_Y(E) \rtimes W$. It is clear that $\sigma \circ \mu' = \mu$. Therefore, we just need to show that the image of $\mu'$ is contained in the subsheaf $\F^1_{\bc}(Y,W)$. This is a local computation. Hence we may assume that $Y = X \times T$, in which case $\sH_{\bc}(Y,W) = \sH_{\bc}(X,W) \boxtimes \dd_T$. Now the claim is clear. 
\end{proof}

The group $T$ acts on $\sH_{\bc}(Y,W)$ and the map $\mu_{\bc}$ is $T$-equivariant. Moreover, a local computation (using the fact that the bundle $Y \rightarrow X$ is locally trivial) shows that the image of $\mf{t}$ is central in $(\pi_{\idot} \sH_{\bc}(Y,W))^T$, hence we may perform quantum Hamiltonian reduction. Recall that we have defined the map $\beta : \mf{t}^* \rightarrow \HH^2(X,\Omega^{1,2}_X)$ in (\ref{eq:betadef}). 

\begin{prop}\label{prop:principalWiso}
Let $\chi \in \mf{t}^*$. We have an isomorphism of sheaves of algebras on $X$,
$$
\sH_{\beta(\chi),\bc}(X,W) \simeq (\pi_{\idot} \sH_{\bc}(Y,W))^T / \langle \{\mu_{\bc}(t) - \chi(t) \ | \ t \in \mf{t} \} \rangle. 
$$
\end{prop}

\begin{proof}
As in the proof of Proposition \ref{lem:chainrule}, let $D = \bigcup_{\bc (w,Z) \neq 0} Z$, $U = X - D$, $E = \pi^{-1}(D)$ and $V = Y - E$. Then the restriction of $\pi$ to $V$ is a principal $T$-bundle $\Pi : V \rightarrow U$ and we have a Cartesian diagram 
$$
\xymatrix{
V \ar@{^{(}->}[r]^j \ar[d]_{\Pi} & Y \ar[d]^{\pi} \\
U \ar@{^{(}->}[r]^k & X
}
$$
Proposition \ref{prop:principaltwist} implies that there is an isomorphism 
\beq{eq:quoiso}
( \Pi_{\idot} \dd_V \rtimes W)^T / \langle \{ \mu'(t) - \chi(t) \ | \ t \in \mf{t} \} \rangle \stackrel{\sim}{\longrightarrow} \dd_U^{\beta(\chi)} \rtimes W. 
\eeq
Recall that $\dd_X^{\beta(\chi)}(D) \rtimes W = k_{\idot}( \dd_U^{\beta(\chi)} \rtimes W)$. Since $k_{\idot} (\Pi_{\idot}  \dd_V \rtimes W)^T = (k_{\idot} (\Pi_{\idot}  \dd_V \rtimes W))^T = (\pi_{\idot} (j_{\idot}  \dd_V \rtimes W))^T$ and $(\pi_{\idot} \sH_{\bc}(Y,W))^T$ is a subalgebra of $(\pi_{\idot} j_{\idot}  \dd_V \rtimes W)^T$, we have a morphism of sheaves $\tau : (\pi_{\idot} \sH_{\bc}(Y,W))^T \rightarrow \dd_X^{\beta(\chi)}(D) \rtimes W$. The isomorphism (\ref{eq:quoiso}) implies that $\langle \{ \mu_{\bc}(t) - \chi(t) \ | \ t \in \mf{t} \} \rangle$ is contained in the kernel of $\tau$. Therefore it suffices to show that $\langle \{ \mu_{\bc}(t) - \chi(t) \ | \ t \in \mf{t} \} \rangle$ is precisely the kernel of $\tau$ and that the image of $\tau$ is $\sH_{\beta(\chi),\bc}(X,W)$. Both of these statements are local. Thus, we may assume without loss of generality that $Y = X \times T$. In this case, both statements reduced to the statement $\dd (T)^T / \langle \{ t - \chi(t) \ | \ t \in \mf{t} \} \rangle \simeq \C$, which is clear. 
\end{proof}

\subsection{} As for differential operators on principal $T$-bundles, see section 2.5 of \cite{BBJantzen}, Proposition \ref{prop:principalWiso} implies that have an equivalence of categories: 

\begin{thm}\label{thm:monodromicequiv}
The functor 
$$
\LMod{(\sH_{\bc}(X, W),T,\chi)} \rightarrow \LMod{\sH_{\beta(\chi),\bc}(Y, W)}, \quad \mm \mapsto (\pi_{\idot} \mm)^T
$$
is an equivalence of categories with quasi-inverse $\ms{N} \mapsto \pi^* \ms{N}$. 
\end{thm}

The above theorem can be extended in the obvious way to the category of weakly $T$-equivariant $\sH_{\bc}(X, W)$-modules with generalized central character $\overline{\chi} \in \mf{t}^* / \mathbb{X}(T)$, as in \cite{BBJantzen}. We leave the details to the interested reader. 

\section{Affinity of Cherednik algebras on projective space}

In this section we prove the main result, which is a criterion for the affinity of Cherednik algebras on $\mathbb{P}(V)$.

\subsection{} Let $V$ be a vector space and $W \subset GL(V)$ a finite group. For each $(s,H) \in \mc{S}(V)$ and $(s,H^*) \in \mc{S}(V^*)$, we fix $\alpha_s \in V^*$ and $\alpha_s^{\vee} \in V$ such that $H = \Ker \alpha_s$ and $H^* = \Ker \alpha_s^{\vee}$, normalized so that $\alpha_s(\alpha_s^{\vee}) = 2$. Let $\Vo = V - \{ 0 \}$ and $\pi : \Vo \rightarrow \mathbb{P}(V)$ the quotient map. It is a principal $T$-bundle, where $T = \Cs$ acts on $V$ by dilations i.e. $t \cdot v = t^{-1} v$ for $t \in T$ and $v \in V$. Since $W$ acts on $V$ it also acts on $\mathbb{P}(V)$. For each $s \in W$, $\codim \mathbb{P}(V)^{s} = 1$ if and only if $s$ is a reflection, in which case $\mathbb{P}(V)^{s} = \mathbb{P}(H) \cup \C \cdot \alpha_s^{\vee}$. 

\begin{lem}\label{lem:alltwisteddiff}
We have $\HH^2(\mathbb{P}(V),\Omega_{\mathbb{P}}^{1,2}) \simeq \C$ and the morphism $\beta$ of (\ref{eq:betadef}) is an isomorphism. 
\end{lem}

\begin{proof}
For each $n \in \Z$, let $\lambda_n$ be the character of $\Cs$ given by $t \mapsto t^n$. Then, $(\pi_{\idot} \mc{O}_{\Vo})^{\lambda_n} \simeq \mc{O}(n)$. This implies that $\beta$ is injective. Therefore, it suffices to show that $\dim \HH^2(\mathbb{P}(V),\Omega_{\mathbb{P}}^{1,2}) = 1$. Since $\mathbb{P}(V)$ can be covered by open sets isomorphic to $\mathbb{A}^{n-1}$ and $H^{i}_{DR}(\mathbb{A}^{n-1}) = 0$ for $i \neq 0$, the algebraic de Rham complex is acyclic. This implies that the map $d \mc{O}_{\mathbb{P}} [-1] \rightarrow \Omega^{1,2}_{\mathbb{P}}$ is a quasi-isomorphism. Therefore, the map $H^1(\mathbb{P}(V), d\mc{O}_{\mathbb{P}}) = \HH^2(\mathbb{P}(V), d\mc{O}_{\mathbb{P}} [-1]) \rightarrow \HH^2(\mathbb{P}(V),\Omega^{1,2}_{\mathbb{P}})$ is an isomorphism. The long exact sequence associated to the short exact sequence 
$$
0 \rightarrow \C_{\mathbb{P}} \rightarrow \mc{O}_{\mathbb{P}} \rightarrow d\mc{O}_{\mathbb{P}} \rightarrow 0
$$
shows that $H^1(\mathbb{P}(V), d\mc{O}_{\mathbb{P}})  \simeq H^2(\mathbb{P}(V), \C_{\mathbb{P}})$ is one-dimensional.  
\end{proof}

Lemma \ref{lem:alltwisteddiff} implies the well-known fact that twisted differential operators on projective space are locally trivial in the Zariski topology. We identify $\HH^2(\mathbb{P}(V),\Omega_{\mathbb{P}}^{1,2})$ with $\C$ such that if $\omega = n \in \Z$, then $\dd_{\mathbb{P}(V)}^{\omega}$ acts on $\mc{O}(n)$. The action of $W$ on $\HH^2(\mathbb{P}(V),\Omega_{\mathbb{P}}^{1,2})$ is trivial, therefore the sheaf $\dd_{\mathbb{P}(V)}^{\omega}$ is $W$-equivariant for all $\omega$. 

\subsection{} When $X = V$, the \textit{rational Cherednik algebra} $\H_{\bc}(V,W)$, as introduced by Etingof and Ginzburg, can be described as an algebra given by generators and relations. Namely, it is the quotient of the skew group algebra $T(V \oplus V^*) \rtimes W$ by the ideal generated by the relations
\begin{equation}\label{eq:rel}
[x,x'] = 0, \qquad [y,y'] = 0, \qquad [y,x] = x(y) - \sum_{s \in \mathcal{S}} \bc(s) \alpha_s(y) x(\alpha_s^\vee) s,
\end{equation}
for all $x,x' \in V^*$ and $y,y' \in V$. Let $x_1, \ds, x_n$ be a basis of $V^*$ and $y_1, \ds, y_n \in V$ the dual basis. The Euler element is 
\begin{align*}
\mathbf{h} & = \sum_{i = 1}^n x_i y_i - \sum_{s \in \mc{S}} \frac{2 \bc(s)}{1 - \lambda_s} s \\
 & = \sum_{i = 1}^n y_i x_i - n + \sum_{s \in \mc{S}} 2 \bc(s) \left( 1 - \frac{1}{1 - \lambda_s} \right) s.
\end{align*}
One can easily check that $[\mathbf{h},x] = x$, $[\mathbf{h},y] = -y$ and $[\mathbf{h},w] = 0$ for all $x \in V^*$, $y \in V$ and $w \in W$. The element $\mathbf{h}$ defines an internal grading on $\H_{\bc}(V,W)$, where $\deg (x) = 1$, $\deg (y) = -1$ and $\deg ( w) = 0$. The $m$th graded piece of $\H_{\bc}(V,W)$ is denoted $\H_{\bc}(V,W)_m$. 

\subsection{Dunkl embedding} 

The open subset $U = V - D$ of $V$ is the complement to the zero locus of $\prod_{s\in \mc{S}(V)} \alpha_s$. For $y\in V$, thought of as a constant coefficient differential operator, the corresponding Dunkl operator $D_y$ equals 
$$
\partial_y+\sum_{s\in\mathcal{S}}\frac{2\mathbf{c}(s)}{1-\lambda_s}\frac{\alpha_s(y)}{\alpha_s}(s-1) \in \Gamma(U, \dd_U \rtimes W).
$$
The presentation of $\H_{\bc}(V,W)$ given above is identified with the Cherednik algebra, defined in terms of Dunkl operators, via the injective algebra homomorphism 
\[ 
\H_{\bc}(V,W) \hookrightarrow \Gamma(U,\dd_U \rtimes W); \quad w \mapsto w, x \mapsto x, y \mapsto D_y, 
\]
for all $w \in W$, $x\in V^*$ and $y\in V$. The image of $\mathbf{h}$ under the Dunkl embedding is
\beq{eq:hDunkl}
\mathbf{h} = \sum_{i = 1}^n x_i \frac{\pa}{\pa x_i} - \sum_{s \in \mc{S}} \frac{2\bc(s)}{1 - \lambda_s}.
\eeq

\subsection{The sheaf of Cherednik algebras on $\mathbb{P}(V)$}

Set $\rho_{\bc} = \sum_{s \in \mc{S}} \frac{2 \bc(s)}{1 - \lambda_s}$. As noted in example 2.20 of \cite{ChereSheaf}, the global sections of $\sH_{\omega,\bc}(\mathbb{P}(V),W)$ are related to $\H_{\bc}(V,W)$ as follows: 

\begin{lem}\label{lem:globalsections}
The global sections $\H_{\omega,\bc}(\mathbb{P}(V),W)$ equals $\H_{\bc}(V,W)_0 / (\mathbf{h} + \rho_{\bc} - \omega)$.
\end{lem}

\begin{proof}
By Proposition \ref{prop:principalWiso} we have a morphism 
$$
\H_{\bc}(V,W)_0 = \H_{\bc}(V,W)^T \rightarrow \H_{\bc}(\Vo,W)^T \rightarrow \H_{\omega,\bc}(\mathbb{P}(V),W). 
$$
Equation (\ref{eq:hDunkl}) implies that the operator $\mathbf{h} + \rho_{\bc} - \omega$ is in the kernel of this map because it is in the kernel of the composite
$$
\H_{\bc}(V,W)_0 \rightarrow \H_{\omega,\bc}(\mathbb{P}(V),W) \rightarrow \sH_{\omega,\bc}(\mathbb{P}(V),W) \hookrightarrow \dd^{\omega}_{\mathbb{P}(V)}(D) \rtimes W.
$$
To prove that $\H_{\bc}(V,W)_0 / (\mathbf{h} + \rho_{\bc} - \omega) \rightarrow \H_{\omega,\bc}(\mathbb{P}(V),W)$ is an isomorphism, we consider the associated graded morphism. We have $\gr_{\mc{F}} \H_{\bc}(V,W)_0 = \C[ x_i y_j \ | \ i,j = 1, \ds, n ] \rtimes W$. We claim that 
\begin{align*}
\gr_{\mc{F}} \H_{\omega,\bc}(\mathbb{P}(V),W) & = \Gamma(\mathbb{P}(V), \pi_{\idot} \mc{O}_{T^* \mathbb{P}(V)} \rtimes W) \\
 & = \left( \C[ x_i y_j \ | \ i,j = 1, \ds, n ]  / \left( \sum_{i = 1}^n x_i y_i \right) \right) \rtimes W.
\end{align*}
The second equality just follows from the usual description of $T^* \mathbb{P}(V)$ as the Hamiltonian reduction of $T^* \Vo = \Vo \times V^*$ with respect to the induced action of $T$. The first equality follows from Theorem \ref{thm:PBW}, once one takes into account that the short exact sequences 
$$
0 \rightarrow \F^{m-1}_{\omega,\bc}(\mathbb{P}(V),W) \rightarrow \F^m_{\omega,\bc}(\mathbb{P}(V),W) \rightarrow (\sym^m \Theta_{\mathbb{P}(V)}) \o W \rightarrow 0
$$
imply by induction that $\mathbb{R}^i \Gamma (\F^m_{\omega,\bc}(\mathbb{P}(V),W)) = 0$ for $i > 0$. Therefore, the filtered morphism $\H_{\bc}(V,W)_0 \rightarrow \H_{\omega,\bc}(\mathbb{P}(V),W)$ is surjective, and hence so too is $\H_{\bc}(V,W)_0 / (\mathbf{h} + \rho_{\bc} - \omega) \rightarrow \H_{\omega,\bc}(\mathbb{P}(V),W)$. On the other hand, the associated graded of $\H_{\bc}(V,W)_0 / (\mathbf{h} + \rho_{\bc} - \omega)$ is a quotient of the algebra $\left( \C[ x_i y_j \ | \ i,j = 1, \ds, n ]  / \left( \sum_{i = 1}^n x_i y_i \right) \right) \rtimes W$.  
\end{proof}

\subsection{}\label{sec:mainthm} Let $\Irr W$ be the set of all isomorphism classes of irreducible $W$-modules. The element 
$$
\mathbf{z} := \sum_{s \in \mc{S}} 2 \bc(s) \left( 1 - \frac{1}{1 - \lambda_s} \right) s = - \mathbf{z}_0 + \sum_{s \in \mc{S}} 2 \bc (s) s. 
$$
belongs to the centre of $\C W$. For each $\lambda \in \Irr W$, let $c_{\lambda}$ be the scalar by which $\mathbf{z}$ acts on $\lambda$ and $d_{\lambda}$ the scalar by which $\mbf{z}_0$ acts on $\lambda$. Set
$$
a_{\lambda} := \rho_{\bc} + c_{\lambda} - n - \omega, \quad b_{\lambda} := \rho_{\bc} - d_{\lambda} - \omega.
$$
The sheaf of algebras $\sH_{\omega,\bc}(\mathbb{P}(V),W)$ is said to be \textit{affine} if the global sections functor $\Gamma$ induces an equivalence of categories  
$$
\Gamma : \LMod{\sH_{\omega,\bc}(\mathbb{P}(V),W)} \iso \LMod{\H_{\omega,\bc}(\mathbb{P}(V),W)}.
$$

\begin{thm}\label{thm:affine}
Let $a_{\lambda}$ and $b_{\lambda}$ be as above. 
\begin{enumerate}
\item The functor $\Gamma$ is exact provided $a_{\lambda} \notin \Z_{\ge 0}$ for all $\lambda \in \Irr W$. 
\item The functor $\Gamma$ is conservative provided $b_{\lambda} \notin \Z_{> 0}$ for all $\lambda \in \Irr W$.
\end{enumerate}
Hence, the sheaf of algebras $\sH_{\omega,\bc}(\mathbb{P}(V),W)$ is affine provided $a_{\lambda} \notin \Z_{\ge 0}$ and $b_{\lambda} \notin \Z_{> 0}$ for all $\lambda \in \Irr W$.
\end{thm}

Our proof of Theorem \ref{thm:affine} follows that of \cite[Theorem 1.6.5]{HTT}.  


\begin{proof}
The category of finitely generated $\H_c(V,W)$-modules that are supported on $\{ 0 \} \subset V$ is denoted $\mc{O}_-$. It is category $\mc{O}$ for the rational Cherednik algebra as studied in \cite{GGOR}. We use basic results from \textit{loc. cit.} without reference. The element $\mathbf{h}$ acts locally finitely on modules in $\mc{O}_-$. The generalized eigenvalues of $\mathbf{h}$ on $M \in \mc{O}_-$ are the \textit{weights} of $M$. Let $\Delta(\lambda)$, for $\lambda \in \Irr W$, denote the Verma modules in $\mc{O}_-$. It is isomorphic to $(\mathrm{Sym} V) \o \lambda$ as a $\mathrm{Sym} V \rtimes (\C W \o \C[\mathbf{h}])$-module. The weights of $\Delta(\lambda)$ are $c_{\lambda} - n - \Z_{\ge 0}$. If $M \in \mc{O}_-$, then there exists a projective module $P \in \mc{O}_-$ and surjection $P \twoheadrightarrow M$. The fact that the module $P$ has a Verma flag implies that the weights of $M$ are contained in $\bigcup_{\lambda \in \Irr W} c_{\lambda} - n - \Z_{\ge 0}$. Therefore, zero is not a generalized eigenvalue of $\mathbf{h} + \rho_{\bc} - \omega$ on $M$ provided $c_{\lambda} + \rho_{\bc} - r - \omega - n \neq 0$ for all $r \in \Z_{\ge 0}$ i.e provided $a_{\lambda} \notin \Z_{\ge 0}$. 

Let $0 \rightarrow \mm_1 \rightarrow \mm_2 \rightarrow \mm_3 \rightarrow 0$ be a short exact sequence in $\Lmod{\sH_{\omega,\bc}(\mathbb{P}(V),W)}$. By Theorem \ref{thm:monodromicequiv}, the terms of the sequence $0 \rightarrow \pi^* \mm_1 \rightarrow \pi^* \mm_2 \rightarrow \pi^* \mm_3 \rightarrow 0$ belong to $\Lmod{(\sH_{\bc}(\Vo, W),T,\omega)}$. Moreover, the sequence is exact because $\pi$ is smooth. Let $j : \Vo \hookrightarrow V$. As noted in Lemma \ref{lem:opendirect}, the sheaves $\mathbb{R}^i j_{\idot} (\pi^* \mm_k)$ for $i \ge 0$ and $k = 1,2,3$ are $\H_{\omega,\bc}(V,W)$-modules. The modules $\mathbb{R}^i j_{\idot} (\pi^* \mm_k)$ are supported on $\{ 0 \}$ for all $i > 0$. Therefore, they belong to the ind-category $\mathrm{Ind} \ \mc{O}_-$. The global sections $\Gamma(\mathbb{P}(V),\mm_k)$ are the element of the $ \Gamma(V,j_{\idot} \pi^* \mm_k)^T$. Therefore the long exact sequence 
$$
0 \rightarrow \Gamma(V, j_{\idot} \pi^* \mm_1) \rightarrow \Gamma(V, j_{\idot} \pi^* \mm_2) \rightarrow \Gamma(V, j_{\idot} \pi^* \mm_3) \rightarrow \Gamma(V,\mathbb{R}^1 j_{\idot}( \pi^* \mm_1)) \rightarrow \cdots 
$$
gives rise to 
$$
0 \rightarrow \Gamma(\mathbb{P}(V),\mm_1) \rightarrow \Gamma(\mathbb{P}(V), \mm_2) \rightarrow \Gamma(\mathbb{P}(V), \mm_3) \rightarrow \Gamma(V,\mathbb{R}^1 j_{\idot}(\pi^* \mm_1))^T \rightarrow \cdots 
$$
The space $\Gamma(V,\mathbb{R}^1 j_{\idot}(\pi^* \mm_1))^T$ can be identified with the space of generalized $\mathbf{h}$-eigenvectors in $\Gamma(V,\mathbb{R}^1 j_{\idot} (\pi^* \mm_1))$ with eigenvalue $\omega - \rho_{\bc}$. But if $a_{\lambda} \notin \Z_{\ge 0}$ for all $\lambda$, then this space is necessarily zero. Hence the sequence $0 \rightarrow \Gamma(\mathbb{P}(V),\mm_1) \rightarrow \Gamma(\mathbb{P}(V),\mm_2) \rightarrow \Gamma(\mathbb{P}(V),\mm_3) \rightarrow 0$ is exact. 

Next we need to show if $b_{\lambda} \notin \Z_{>0}$ for all $\lambda \in \Irr W$ then $\Gamma$ is conservative i.e. $\Gamma (\mathbb{P}(V), \ms{M}) = 0$ implies that $\ms{M} = 0$. Assume that $\ms{M} \neq 0$. Since $\pi$ is smooth and surjective, it is faithfully flat and $\pi^* \ms{M} = 0$ implies that $\ms{M} = 0$. Hence $\pi^* \ms{M} \neq 0$. Since $\pi^* \mm$ is $(T,\omega)$-monodromic, the Euler element $\mathbf{h}$ acts semi-simply on $\Gamma(V, j_{\idot} \pi^* \ms{M})$, hence it decomposes as
$$
\Gamma(V, j_{\idot} \pi^* \ms{M}) = \bigoplus_{\alpha \in \Z} \Gamma(V, j_{\idot} \pi^* \ms{M})_{\alpha + \omega - \rho_{\bc}}. 
$$
There is some $\alpha \in \Z$ for which $\Gamma(V, j_{\idot} \pi^* \ms{M})_{\alpha + \omega - \rho_{\bc}} \neq 0$. We first assume that $\alpha > 0$. Choose $0 \neq m \in \Gamma(V,j_{\idot} \pi^* \ms{M})_{\alpha + \omega - \rho_{\bc}}$. Since the space $\Gamma(V, j_{\idot} \pi^* \ms{M})_{\alpha + \omega - \rho_{\bc}}$ is a $W$-module, we may assume that $m$ lies in some irreducible $W$-isotypic component (of type $\lambda$ say) of $\Gamma(V, j_{\idot} \pi^* \ms{M})_{\alpha + \omega - \rho_{\bc}}$. We claim that there is some $y$ such that $y \cdot m \neq 0$. Assume not, then $\mathbf{h} \cdot m = -d_\lambda m$. Hence $-d_\lambda = \alpha + \omega - \rho_{\bc}$ i.e. $b_{\lambda} = \rho_{\bc} -  d_{\lambda} - \omega = \alpha \in \Z_{> 0}$, contradicting our assumption on $b_{\lambda}$. Thus $y \cdot m \neq 0$. But $y \cdot m \in \Gamma(V, j_{\idot} \pi^* \ms{M})_{\alpha - 1 +\omega - \rho_{\bc}}$ so eventually we get a non-zero vector in $\Gamma(V, j_{\idot} \pi^* \ms{M})_{\omega - \rho_{\bc}}$ as required. Now assume that $\alpha < 0$. If $m \in \Gamma(\Vo, \pi^* \ms{M})_{\alpha + \omega - \rho_{\bc}}$ is a non-zero section, then the support of $m$ is not contained in $\{ 0 \}$. On the other hand, if $x \cdot m = 0$ for all $x \in V^*$, then $\mathrm{Supp} (m) \subset \{ 0 \}$ and hence $m = 0$. Hence $m \neq 0$ implies that there exists some $x \in V^*$ such that $x \cdot m \neq 0$. Repeating this argument, we eventually conclude that $\Gamma(\Vo, \pi^* \ms{M})_{\omega - \rho_{\bc}} \neq 0$. 
\end{proof}


When $W $ is trivial, Theorem \ref{thm:affine} says that $\mathbb{P}(V)$ is $\dd^{\omega}$-affine provided $\omega \notin \{ - n , - n - 1, \ds \}$, which equals the set of all $\omega \in \mc{A} \cup \mc{E}$ of \cite[Theorem 6.1.3]{ToricVdB}.  
 
\begin{rem}
The action of $W$ on $V$ induces an action of $W$ on all the partial flag manifolds $GL(V) / P$, where $P$ is a parabolic of $GL(V)$. However, one can check that there are reflections in $(GL(V) / P, W)$ if and only if $GL(V) / P = \mathbb{P}(V)$ or $GL(V) / P$ is the Grassmannian of codimension one subspaces in $V$. 
\end{rem}

\section{A local presentation of the Cherednik algebra}

In this section we give a local presentation of the sheaf of Cherednik algebras. 

\subsection{} In this section only, we make the following assumptions 
\begin{itemize}
\item For each $(w,Z) \in \mc{S}(X)$, there exists a globally defined function $f_Z$ such that $Z = V(f_Z)$. 
\item All Picard algebroids considered can be trivialized in the Zariski topology. 
\end{itemize}
We fix a choice of functions $f_Z$.  

\subsection{The $\KZ$-connection}

Recall that $U = X - \bigcup_{(w,Z)}  Z$, where the union is over all $(w,Z)$ in $\mc{S}(X)$. Since we have fixed a choice of defining equations of the hypersurfaces $Z$, it is possible to write down a $\KZ$-connection on $U$. 

\begin{defn}
The Knizhnik-Zamolodchikov connection on $U$, with values in $\mc{O}_{U} \o \C W$, is defined to be  
$$
\omega_{X,\bc} = \sum_{(w,Z) \in \mc{S}(X)} \frac{2 \bc(w,Z)}{1 - \lambda_{w,Z}} (d \log f_{Z}) \o s. 
$$
\end{defn}


The $\KZ$-connection behaves well under melys morphisms. 

\begin{lem}
Let $\map : Y \rightarrow X$ be a \textit{surjective} morphism, melys for $\bc$. Then, $\map^* \omega_{Y,\bc} = \omega_{X,\map^* \bc}$. 
\end{lem}

\begin{proof}
The fact that $\map$ is surjective implies that $\map^* f_Z$ is not a unit for all $(w,Z) \in \mc{S}_{\bc}(X)$. Then, the lemma follows from equation \ref{eq:pullbackform}, since the term $h$ there can be chosen to be zero.  
\end{proof}

\subsection{} Fix $\omega \in \HH^2(X,\Omega^{1,2}_X)^W$, trivializable in the Zariski topology. For $(w,Z) \in \mc{S}(X)$ and $\nu_1, \nu_2 \in \mc{P}^{\omega}$, define 
$$
\Xi_{Z}^w(\nu_1,\nu_2) :=  (d \log f_Z \wedge \sigma(\nu_1)) (w(\nu_2) - \nu_2) - (d \log f_Z \wedge \sigma(\nu_2)) (w (\nu_1) - \nu_1)
$$
in $\mc{P}^{\omega}(D)$. 

\begin{lem}
Let $(w, Z) \in \mc{S}(X)$, $g \in \mc{O}_X$ and $\nu_1, \nu_2 \in \mc{P}^{\omega}$. Then, 
$$
(d \log f_Z \wedge \nu) (w(g) - g) \in \mc{O}_X \quad \textrm{and} \quad \Xi^{w}_{Z}(\nu_1,\nu_2) \in \mc{P}^{\omega}.
$$ 
\end{lem}

\begin{proof}
If $g \in \mc{O}_X$ and $\nu \in \mc{P}^{\omega}$ then $(d \log f_Z \wedge \sigma(\nu)) (w(g) - g) \in \mc{O}_X$ because $w(g) - g \in I(Z)$. The second claim is that 
$$
\frac{\sigma(\nu_1)(f_Z)}{f_Z}(w(\nu_2) - \nu_2) - \frac{\sigma(\nu_2)(f_Z)}{f_Z} (w (\nu_1) - \nu_1) \in \mc{P}^{\omega}. 
$$
The statement is local and clearly true in a neighborhood of any point of $X - Z$. Therefore, we may assume that we have fixed a point $x \in Z$. Choose a small, affine $w$-stable open subset $U$ of $X$ with coordinate system $x_1, \ds, x_n$ such that $w (x_1) = \zeta x_1$ and $w(x_i ) = x_i $ for $i \neq 1$. Moreover, since we have assumed that the Picard algebroid $\mc{P}^{\omega}$ trivializes in the Zariski topology, we may assume that $\mc{P}^{\omega} |_U = \mc{O}_U \oplus \Theta_U$. There exists some unit $u \in \Gamma(U,\mc{O}_X)$ such that $f_Z = u x_1$. The statement is clear if either of $\nu_1$ or $\nu_2$ is in $\Gamma(U,\mc{O}_X)$. Thus, without loss of generality, $\nu_1,\nu_2 \in \Gamma(U,\Theta_X)$. Expanding, 
$$
\Xi^{w}_{Z}(\nu_1,\nu_2) = \frac{\nu_1(x_1)}{x_1}(w(\nu_2) - \nu_2) - \frac{\nu_2(x_1)}{x_1} (w (\nu_1) - \nu_1) + h 
$$
for some $h \in \Gamma(U,\Theta_X)$. There are $f_i, g_i \in \Gamma(U,\mc{O}_X)$ such that $\nu_1 = \sum_{i = 1}^n f_i \frac{\pa}{\pa x_i}$ and $\nu_2 = \sum_{i = 1}^n g_i \frac{\pa}{\pa x_i}$. We have 
\begin{align*}
\frac{\nu_1(x_1)}{x_1}(w(\nu_2) - \nu_2) = & \sum_{i,j =1}^n f_i x_1^{-1} \frac{\pa x_1}{\pa x_i} \left( w(g_j) \frac{\pa}{\pa w(x_j)} - g_j \frac{\pa}{\pa x_j} \right) \\
 & = \sum_{j = 1}^n f_1 x_1^{-1}  \left( w(g_j) \frac{\pa}{\pa w(x_j)} - g_j \frac{\pa}{\pa x_j} \right) \\
 & = \sum_{j = 1}^n f_1 x_1^{-1}  \left( ( w(g_j) - g_j) \frac{\pa}{\pa w(x_j)} + g_j \left( \frac{\pa}{\pa w(x_j)} - \frac{\pa}{\pa x_j} \right) \right) \\
 & = f_1 g_1 x_1^{-1} \zeta \frac{\pa}{\pa x_1} + \sum_{j = 1}^n f_1 x_1^{-1}  \left( ( w(g_j) - g_j) \frac{\pa}{\pa w(x_j)} \right).
\end{align*}
Thus, if we define $h_1$ to be $\sum_{j = 1}^n f_1 x_1^{-1}  \left( ( w(g_j) - g_j) \frac{\pa}{\pa w(x_j)} \right)$, which belongs to $\Gamma(U,\mc{P}^{\omega})$!, we have 
\begin{align*}
\frac{\nu_1(x_1)}{x_1}(w(\nu_2) - \nu_2) - \frac{\nu_2(x_1)}{x_1}(w(\nu_1) - \nu_1) & = f_1 g_1 x_1^{-1} \zeta \frac{\pa}{\pa x_1} + h_1 - f_1 g_1 x_1^{-1} \zeta \frac{\pa}{\pa x_1} - h_2, \\
 & = h_1 - h_2,
\end{align*}
which belongs to $\Gamma(U,\Theta_X)$. 
\end{proof}

\subsection{} We define the sheaf of algebras $\mc{U}_{\omega,\bc}(X,W)$ to be the quotient of $T \mc{P}^{\omega} \rtimes W$ by the relations
\begin{align}
\nu \o g - g \o \nu & = \sigma(\nu)(g) + \sum_{(w,Z)} \frac{2 \bc (w,Z)}{1 - \lambda_{w,Z}} (d \log f_Z \wedge \sigma(\nu)) (w(g) - g) w, \label{eq:Dcom}\\
\nu_1 \o \nu_2 - \nu_2 \o \nu_1 & = [\nu_1, \nu_2] +  \sum_{(w,Z)} \frac{2 \bc (w,Z)}{1 - \lambda_{w,Z}} \ \Xi_{Z}^w(\nu_1, \nu_2) w, \label{eq:fieldcom}
\end{align}
for all $\nu, \nu_1, \nu_2 \in \mc{P}^{\omega}_X$ and $g \in \mc{O}_X$, and the relation\footnote{Recall from definition \ref{defn:Picard} that $1_{\mc{P}}$ is defined to be the image of $1 \in \mc{O}_X$ under the map $i : \mc{O}_X \rightarrow \mc{P}$.} $1_{\mc{P}} = 1$. 

\begin{rem}
When $X = V$ is a vector space and $\nu_1, \nu_2 \in V$ are constant coefficient vector fields, the right-hand side of (\ref{eq:fieldcom}) is zero and we get the usual relations of the rational Cherednik algebra. 
\end{rem}

\begin{prop}\label{prop:presentation}
The map $\nu \mapsto D_{\nu}$, $w \mapsto w$ for $\nu \in \mc{P}^{\omega}$ and $w \in W$ defines an isomorphism $\mc{U}_{\omega,\bc}(X,W) \iso \sH_{\omega,\bc}(X,W)$ if and only if the $\KZ$-connection is flat.
\end{prop}

\begin{proof}
The proof is a direct calculation. It is straight-forward to see that relation (\ref{eq:Dcom}) always holds in $\sH_{\omega,\bc}(X,W)$. Therefore, we just need to check that relation (\ref{eq:fieldcom}) holds for Dunkl operators in $\sH_{\omega,\bc}(X,W)$ if and only if the $\KZ$-connection is flat. Let $\nu_1, \nu_2 \in \mc{P}^{\omega}_X$ and $D_{\nu_1}, D_{\nu_2}$ the corresponding Dunkl operators. We need to calculate the right hand side of 
$$
[D_{\nu_1}, D_{\nu_2}] = \left[\nu_1 + \sum_{(w,Z)} \frac{2 \bc(w,Z)}{1 - \lambda_{w,Z}} \frac{\sigma(\nu_1)(f_Z)}{f_Z} (w - 1), \nu_2 + \sum_{(w,Z)} \frac{2 \bc(w,Z)}{1 - \lambda_{w,Z}} \frac{\sigma(\nu_2)(f_Z)}{f_Z} (w - 1) \right]. 
$$
We have 
\begin{align*}
\left[\frac{\sigma(\nu_1)(f_Z)}{f_Z}(w - 1), \nu_2\right] = & \frac{\sigma(\nu_2) \circ \sigma(\nu_1) (f_Z)}{f_Z} (w - 1) - \frac{\sigma(\nu_1)(f_Z) \sigma(\nu_2)(f_Z)}{f_Z^2} (w - 1) \\
 & + \frac{\sigma(\nu_1)(f_Z)}{f_Z} (w(\nu_2) - \nu_2) w, 
\end{align*}
and hence $\sum_{(w,Z)} \frac{2 \bc(w,Z)}{1 - \lambda_{w,Z}} \left( \left[ \frac{\nu_1(f_Z)}{f_Z}(w - 1), \nu_2\right] + \left[\nu_1, \frac{\nu_2(f_Z)}{f_Z}(w - 1)\right]\right)$ equals 
$$
\sum_{(w,Z)} \frac{2 \bc (w,Z)}{1 - \lambda_{w,Z}} \left( \frac{[\nu_1, \nu_2](f_Z)}{f_Z} (w - 1) + \frac{\nu_1(f_Z)}{f_Z}(w(\nu_2) - \nu_2) w  - \frac{\nu_2(f_Z)}{f_Z} (w (\nu_1) - \nu_1) w \right). 
$$
Also, $\left[ -\frac{\nu_1(f_Z)}{f_Z} w_1 , \frac{\nu_2(f_{Z'})}{f_{Z'}}\right] + \left[ \frac{\nu_1(f_Z)}{f_Z} , -\frac{\nu_2(f_{Z'})}{f_{Z'}} w_2 \right]  + \left[ \frac{\nu_1(f_Z)}{f_Z} w_1 , \frac{\nu_2(f_{Z'})}{f_{Z'}} w_2\right]$ equals  
\begin{align*}
& -\frac{\nu_1(f_Z)}{f_Z} w_1 \left(\frac{\nu_2(f_{Z'})}{f_{Z'}}\right) (w_1 - 1) +  \frac{\nu_2(f_{Z'})}{f_{Z'}}   w_2 \left(\frac{\nu_1(f_{Z})}{f_{Z}}\right) (w_2 - 1)\\
& + \frac{\nu_1(f_Z)}{f_Z} w_1 \left(\frac{\nu_2(f_{Z'})}{f_{Z'}}\right)  w_1 w_2 - \frac{\nu_2(f_{Z'})}{f_{Z'}}   w_2 \left(\frac{\nu_1(f_{Z})}{f_{Z}}\right) w_2 w_1. 
\end{align*} 
Combining the above equations, one sees that relation (\ref{eq:fieldcom}) holds for Dunkl operators in $\sH_{\omega,\bc}(X,W)$ if and only if  
$$
\sum_{(w_1,Z),(w_2,Z')} \frac{4 \bc (w_1,Z) \bc(w_2,Z')}{(1 - \lambda_{w_1,Z})(1 - \lambda_{w_2,Z'})} \left( \frac{\nu_2(f_{Z})}{f_{Z}} \frac{\nu_1(f_{Z'})}{f_{Z'}} - \frac{\nu_1(f_{Z})}{f_{Z}} \frac{\nu_2(f_{Z'})}{f_{Z'}} \right) w_1 w_2 = 0.
$$
Since the left hand side equals  
$$
\left( \sum_{(w_1,Z),(w_2,Z')} \frac{4 \bc (w_1,Z) \bc(w_2,Z')}{(1 - \lambda_{w_1,Z})(1 - \lambda_{w_2,Z'})} (d \log f_Z \wedge d \log f_{Z'}) \o w_1 w_2 \right) (\nu_2, \nu_1),
$$
it will be zero for all $\nu_1, \nu_2$ if and only if the meromorphic two-form inside the bracket is zero. But this two-form is the curvature $\omega_{X,\bc} \wedge \omega_{X,\bc}$ of the $\KZ$-connection. 
\end{proof}

Proposition \ref{prop:presentation} implies that, when the $\KZ$-connection is flat, the algebra $\mc{U}_{\omega,\bc}(X,W)$ is, up to isomorphism, independent of the choice of functions $f_Z$. 

\section{Appendix: TDOs}

In the appendix we summarize the facts we need about twisted differential operators, following \cite{BBJantzen} and \cite{KasAsterisque}.

\subsection{Twisted differential operators}

It is most natural to realize sheaves of twisted differential operators as a quotient of the enveloping algebra of a Picard algebroid. 

\begin{defn}\label{defn:Picard}
The $\mc{O}_X$-module $\ms{L}$ is called a \textit{Lie algebroid} if there exists a bracket $[ - , - ] : \ms{L} \o_{\C_X} \ms{L} \rightarrow \ms{L}$ and morphism of $\mc{O}_X$-modules $\sigma : \ms{L} \rightarrow \Theta_X$ (the anchor map) such that $(\ms{L},[ - , - ])$ is a sheaf of Lie algebras with the anchor map being a morphism of Lie algebras and, for $l_1, l_2 \in \ms{L}$ and $f \in \mc{O}_X$,  
$$
[l_1 , f l_2 ] = f [l_1, l_2] + \sigma(l_1)(f) l_2. 
$$
If, moreover, there exists a map $i : \mc{O}_X \rightarrow \ms{L}$ of $\mc{O}_X$-modules such that the sequence 
$$
0 \rightarrow \mc{O}_X \rightarrow \ms{L} \rightarrow \Theta_X \rightarrow 0 
$$
is exact and $i(1) := 1_{\ms{L}}$ is central in $\ms{L}$, then $\ms{L}$ is called a \textit{Picard algebroid}. 
\end{defn}

As in \cite{BBJantzen}, we denote by $\Omega_X^{1,2}$ the two term subcomplex $\Omega_X^1 \stackrel{d}{\longrightarrow} (\Omega_X^2)^{\mathrm{cl}}$, concentrated in degrees $1$ and $2$, of the algebraic de-Rham complex of $X$.

\begin{prop}
The Picard algebroids on $X$ are parameterized up to isomorphism by $\HH^2(X,\Omega^{1,2}_X)$. 
\end{prop}

Given $\omega \in \HH^2(X,\Omega_X^{1,2})$, the corresponding Picard algebroid is denoted $\mc{P}^{\omega}_X$. Associated to $\mc{P}^{\omega}_X$ is $\dd_{X}^{\omega}$, the sheaf of differential operators on $X$ with twist $\omega$. It is the quotient of the enveloping algebra $\mc{U}(\mc{P}^{\omega}_X)$ of $\mc{P}^{\omega}_X$ by the ideal generated by $1_{\mc{P}^{\omega}_X} - 1$.  

\begin{defn}
A module for the Picard algebroid $\mc{P}$ is a quasi-coherent $\mc{O}_X$-module $\mm$ together with a map $- \cdot - : \mc{P} \o_{\C_X} \mm \rightarrow \mm$ such that $i(f) \cdot m = f m$ and $[p,q] \cdot m = p \cdot (q \cdot m) - q \cdot (p \cdot m)$ for all $p, q \in \mc{P}, m \in \mm$ and $f \in \mc{O}_X$. 
\end{defn}

There is a natural equivalence between the category of $\mc{P}^{\omega}$-modules and the category of $\dd^{\omega}$-modules.

\subsection{Functorality} We recall from section 2.2 of \cite{BBJantzen} the functionality properties of Picard algebroids and twisted differential operators. Fix a morphism $\map : Y \rightarrow X$. Let $\mc{P}_X$ be a Picard algebroid on $X$ and $\mc{P}_Y$ a Picard algebroid on $Y$. 

\begin{defn}\label{defn:fmorphism}
A $\map$-morphism $\gamma : \mc{P}_Y \rightarrow \mc{P}_X$ is an $\mc{O}_Y$-linear map $\gamma : \mc{P}_Y \rightarrow \map^* \mc{P}_X$ such that for any section $p \in \mc{P}_Y$, $\gamma(p) = \sum_i g^i \o q^i$ with $g_i \in \mc{O}_Y$ and $q_i \in \map^{-1} \mc{P}_X$, we have 
$$
\gamma([p_1,p_2]) = \sum_{i,j} g_1^i g_2^j \o [q_1^i,q_2^j] + \sum_{j} \sigma(p_1)(g_2^j) \o q_2^j - \sum_i \sigma(p_2) (g_1^i) \o q_1^i,
$$
and $\sigma(n)(f^* g) = \sum_i g^i \map^*(\sigma(q^i)(g))$ for all $g \in \map^{-1} \mc{O}_X$.
\end{defn}

The first fundamental theorem on differential forms, \cite[Theorem 25.1]{MatCom}, implies that there is a morphism of sheaves $\map^{-1} \Omega_X^1 \rightarrow \Omega_Y^1$. This extends to a morphism of complexes $\map^{-1} \Omega_X^{\idot} \rightarrow \Omega_Y^{\idot}$ and $\map^{-1} \Omega_X^{1,2} \rightarrow \Omega^{1,2}_Y$. By functorality of hypercohomology, we get a map $\map^* : \HH^2(X,\Omega_X^{1,2}) \rightarrow \HH^2(Y,\Omega^{1,2}_Y)$. For $\omega \in \HH^2(X,\Omega_X^{1,2})$, let $\mc{P}_X^{\omega}$ be the corresponding Picard algebroid and $\mc{P}_Y$ the fiber product $\map^* \mc{P}_X^{\omega} \times_{\map^* \Theta_X} \Theta_Y$, where $\map^* \mc{P}_X^{\omega} \rightarrow \map^* \Theta_X$ is the anchor map and $\Theta_Y \rightarrow \map^* \Theta_X$ is $d\map$. 

\begin{lem}\label{lem:fmorphism}
The sheaf $\mc{P}_Y$ is a Picard algebroid, $\psi : \mc{P}_Y \rightarrow \map^* \mc{P}_X$ is a $\map$-morphism and we have an isomorphism of Picard algebroids $\mc{P}_Y \simeq \mc{P}_Y^{\map^* \omega}$. 
\end{lem}

Thus, by definition, the diagram 
$$
\xymatrix{ 
0 \ar[r] & \mc{O}_Y \ar[r] \ar[d] & \mc{P}_Y^{\map^* \omega} \ar[d]_{\psi} \ar[r]^{\sigma_Y} & \Theta_Y  \ar[d] \ar[r] & 0 \\
 & \mc{O}_Y = \map^* \mc{O}_X \ar[r] & \map^* \mc{P}_X^{\omega}  \ar[r]^{\sigma_X}  & \map^* \Theta_X &  
}
$$
commutes. The projection $\mc{P}_Y^{\map^* \omega} \rightarrow \map^* \mc{P}_X^{\omega}$ extends to a morphism $\dd_Y^{\map^* \omega} \rightarrow \map^* \dd_X^{\omega}$, making $\map^* \dd_X^{\omega}$ a left $\dd_Y^{\map^* \omega}$-module. Let $\ms{M}$ be a left $\dd_X^{\omega}$-module. Since $\map^* \ms{M} = \map^* \dd_X^{\omega} \o_{\map^{-1} \dd_X^{\omega}} \map^{-1} \ms{M}$, we have 

\begin{prop}
For any $\ms{M} \in \LMod{\dd_X^{\omega}}$, the sheaf $\map^* \ms{M}$ is a $\dd_Y^{\map^* \omega}$-module. 
\end{prop}

\begin{rem}\label{rem:etalepullback}
If $\map$ is etal\'e, then $d \map : \Theta_Y \rightarrow \map^* \Theta_X$ is an isomorphism. Therefore, the projection $\mc{P}_Y^{\map^* \omega} \rightarrow \map^* \mc{P}_X^{\omega}$ is also an isomorphism and, in this case, the isomorphism $\gamma : \dd_Y^{\map^* \omega} \rightarrow \map^* \dd_X^{\omega}$ of left $\dd_Y^{\map^* \omega}$-modules is actually an algebra isomorphism (in particular, $\map^* \dd_X^{\omega}$ is a sheaf of algebras).
\end{rem}

\subsection{Monodromic $\dd$-modules}

Let $T$ be a torus i.e. a product of copies the multiplicative group $\Cs$. The Lie algebra of $T$ is denoted $\mf{t}$. Let $\pi : Y \rightarrow X$ be a principal $T$-bundle, with $X$ smooth. A common way of constructing sheaves of twisted differential operators on $X$ is by quantum Hamiltonian reduction. Let $\mu : \mf{t} \rightarrow \dd_Y$ be the differential of the action of $T$ on $Y$. Since $\dd_Y$ is a $T$-equivariant sheaf, there is a stalk-wise action of $T$ on $\pi_{\idot} \dd_Y$. The map $\mu$ is $T$-equivariant and, since $T$ acts trivially on $\mf{t}$, $\mu$ descends to a map $\mf{t} \rightarrow (\pi_{\idot} \dd_Y)^T$. The image of $\mu$ is central. Given a character $\chi : \mf{t} \rightarrow \C$, let 
$$
\dd_{X,\chi} := (\pi_{\idot} \dd_Y)^T / \langle \{\mu(t) - \chi(t) \ | \ t \in \mf{t} \} \rangle.
$$
Let $\mathbb{X}(T)$ be the lattice of characters of $T$. By differentiation, we may identify $\mathbb{X}(T)$ with a lattice in $\mf{t}^*$ such that $\mathbb{X}(T) \o_{\Z} \C = \mf{t}^*$. Given $\lambda \in \mathbb{X}(T)$, the sheaf of $\lambda$-semi-invariant sections $(\pi_{\idot} \mc{O}_Y)^{\lambda}$ is a line bundle on $X$. Thus, we have a map $\mathbb{X}(T) \rightarrow H^1(X,\mc{O}_X^{\times})$. Composing this with the map $\mc{O}_X^{\times} \stackrel{d \log}{\longrightarrow} \Ker (d : \Omega_X^1 \rightarrow \Omega^2_X) \subset \Omega_X^{1,2}$ gives a map 
$$
\beta_{\Z} :  \mathbb{X}(T) \rightarrow H^1(X,\mc{O}_X^{\times}) \stackrel{d \log}{\longrightarrow}  \HH^2(X,\Omega_{X}^{1,2})
$$
of $\Z$-modules. Extending scalars, we get a map 
\beq{eq:betadef}
\beta : \mf{t}^* \rightarrow \HH^2(X,\Omega_X^{1,2}).
\eeq 

\begin{prop}\label{prop:principaltwist}
The sheaf of algebras $\dd_{X,\chi}$ is a sheaf of twisted differential operators, isomorphic to $\dd_X^{\beta(\chi)}$. 
\end{prop}

\small{

\bibliographystyle{plain}

\def\cprime{$'$} \def\cprime{$'$} \def\cprime{$'$} \def\cprime{$'$}
  \def\cprime{$'$} \def\cprime{$'$} \def\cprime{$'$} \def\cprime{$'$}
  \def\cprime{$'$} \def\cprime{$'$} \def\cprime{$'$} \def\cprime{$'$}

}

\end{document}